\documentclass[11pt,a4paper]{amsart}
\usepackage{amssymb,xspace}
\usepackage{amstext}
\usepackage{tikz}
\usetikzlibrary{arrows,decorations.pathmorphing,backgrounds,fit,positioning,shapes.symbols,chains}
\theoremstyle{plain}
\usepackage{amsbsy,amssymb,amsfonts,latexsym}

\usepackage{enumerate}
\usepackage{graphics}

\date{\today}
\title{Composition operators on the Hardy space of the tridisc}
\author{Fr\'ed\'eric Bayart}
\address{Laboratoire de Math\'ematiques Blaise Pascal UMR 6620 CNRS, Universit\'e Clermont Auvergne, Campus universitaire C\'ezeaux, 3 place Vasarely, 63178 Aubière Cedex, France.}
\email{frederic.bayart@uca.fr}



\newcommand{\veps}{\varepsilon}

\def\RR{\mathbb R}

\def\NN{\mathbb N}

\def\TT{\mathbb T}
\def\DD{\mathbb D}

\def\CC{\mathbb C}

\def\Imm{\mathrm{Im}}
\def\Ree{\mathrm{Re}}

\def\ovd{\overline{\delta}}

\newcommand{\vect}{\textrm{span}}

\newtheorem{theorem}{Theorem}[section]

\newtheorem{lemma}[theorem]{Lemma}

\newtheorem{corollary}[theorem]{Corollary}

{\theoremstyle{definition}}
{\theoremstyle{definition}}

{\theoremstyle{definition}\newtheorem{example}[theorem]{Example}}

{\theoremstyle{definition}}

{\theoremstyle{definition}}

{\theoremstyle{definition}\newtheorem{remark}[theorem]{Remark}}

\begin{document}

\begin{abstract}
We give a necessary and sufficient condition for a holomorphic self-map $\phi$ of the tridisc to induce a bounded composition operator on the associated Hardy space.
This condition depends on the behaviour of the first and the second derivative of the symbol at boundary points. 
We also discuss compactness of composition operators on the bidisc and the tridisc.
\end{abstract}

\maketitle

\section{Introduction}

Let $\mathcal U$ be a domain in $\CC^d$, let $X$ be a Banach space of holomorphic functions on $\mathcal U$ and let $\phi:\mathcal U\to\mathcal U$
be holomorphic. The composition operator with symbol $\phi$ is defined by $C_\phi(f)=f\circ\phi,$ $f\in X$. The first question to solve when studying composition operators is that of continuity: for which symbols $\phi$ do $C_\phi$ induce a bounded composition operator on $X$? When $X$ is the Hardy space or a weighted Bergman space of the unit disc $\DD,$ the answer is easy: by the Littlewood
subordination principle, this is always the case. 

In several variables, this statement does not hold, and there are simple examples of self-maps of the euclidean ball (resp. of the polydisc)
which do not induce a bounded composition operator on the corresponding Hardy space (see for instance \cite{CSW84,CoMcl95}). Therefore, it is interesting to find a characterization of the symbols
$\phi$ inducing a bounded operator $C_\phi$, adding if necessary some regularity conditions on the symbol. 
For the ball, a very satisfying answer has been provided by Wogen in \cite{Wo88}: he characterized the continuity of $C_\phi$ acting 
on $H^2(\mathbb B_d)$ or $A^2_\alpha(\mathbb B_d)$ provided $\phi$ extends to a $\mathcal C^3$-function up to the boundary of ${\mathbb B_d}$. His characterization involves the first and second order derivatives of $\phi$ at any $\xi\in\partial \mathbb B_d$ such that $\phi(\xi)\in\partial \mathbb B_d.$ As a consequence, he got that continuity on $H^2(\mathbb B_d)$ is equivalent to continuity
on $A_\alpha^2(\mathbb B_d)$. 

In this paper we investigate the case of the polydisc, more precisely of the tridisc $\DD^3.$ 
It appears that characterizing the continuity of composition operators on $H^2(\DD^d)$ is a more difficult task than on $H^2(\mathbb B_d)$.
After a first attempt in \cite{KSZ08}, where a necessary condition is given,
the continuity problem was solved in \cite{BAYPOLY} on $A^2(\DD^2)$ for symbols regular up to the boundary,
and on $H^2(\DD^d)$ or on any $A^2_\alpha(\DD^d)$, $\alpha>-1,$ provided $\phi$ is an affine symbol.
The case of the Hardy space requires a careful and specific argument since the method of \cite{BAYPOLY} goes through Carleson measures
and it is well known that Carleson measures on the Hardy space of the polydisc are difficult objects.

The recent paper \cite{Ko22} solves the continuity problem for $H^2(\DD^2)$ and for $A^2(\DD^3)$. 
Our main theorem (see Theorem \ref{thm:main} below) gives a characterization of the symbols $\phi:\DD^3\to\DD^3$, regular up to the boundary, such that $C_\phi\in\mathcal L(H^2(\DD^3))$. As in Wogen's paper, this characterization involves conditions on the first and the second order derivatives of $\phi$, whereas the characterization on $H^2(\DD^2)$ or $A^2(\DD^3)$ just involve conditions on the first derivative of the symbol. 

Theorem \ref{thm:main} needs some terminology that will be introduced in Section \ref{sec:statement} to be stated but let us give a flavour of it.
One of the main difficulty of working in the polydisc is that we have to play with its two boundaries: its topological boundary and its distinguished boundary. For $\xi\in \TT^3$
such that $\phi(\xi)\notin\DD^3,$ we will discuss the properties of $\phi_I$ at $\xi,$ 
where $I\subset\{1,2,3\}$ is maximal with $\phi_I(\xi)\in\TT^{|I|},$ and $\phi_I=(\phi_{i_1},\dots,\phi_{i_p})$ if $I=\{i_1,\dots,i_p\}.$ In particular, 
the case $|I|=2$ will require a very careful analysis.

Although our characterization is not very easy to state, we will provide several examples showing that it is effective: 
given a concrete $\phi\in\mathcal O(\DD^3,\DD^3)\cap\mathcal C^3(\overline{\DD^3}),$
where $\mathcal O(\mathcal U,\mathcal V)$ denotes the set of holomorphic functions between the two open sets $\mathcal U\subset\CC^d$
and $\mathcal V\subset \CC^{d'}$, we will be able to decide whether $C_\phi\in\mathcal L(H^2(\DD^3))$. 

Beyond continuity, the next question to handle is that of compactness. It seems that it was not considered at all for $H^2(\DD^d)$. We will
first give a sufficient condition on $\phi\in\mathcal O(\DD^d,\DD^d)$ so that $C_\phi$ induces a compact operator on $H^2(\DD^d)$,
for any $d\geq 1$. We will then use this condition to study in depth the cases $d=2$ and $d=3.$

\medskip

{\sc Organization of the paper.} Section 2 contains the terminology needed to state our main theorem, that is Theorem \ref{thm:main}. Section 3 is devoted to some useful results
which will be used throughout the paper. The proof of Theorem \ref{thm:main} is divided into Sections 4 to 6: Sections 4 and 5 are devoted to volume estimates of preimages of Carleson windows by $\phi$ at boundary points with order of contact equal to $2.$ 
These estimates are used to finish the proof of Theorem \ref{thm:main} in Section 6. Several examples are given in Section 7 to illustrate the different cases which might happen whereas Section 8 deals with compactness issues.

\medskip

{\sc Notations.} Throughout this paper, we shall use the following notations. For $f,g:A\to \RR,$ we shall write $f\lesssim g$
provided there exists $C>0$ such that, for all $x\in A,$ $f(x)\leq Cg(x)$. For $\phi:\TT^d\to A$ and 
$\theta\in\RR^d,$ we shall write $\phi(\theta)$ for $\phi(e^{i\theta_1},\dots,e^{i\theta_d})$ and also sometimes $e^{i\theta}=(e^{i\theta_1},\dots,e^{i\theta_d})$.
The Lebesgue measure on $\RR^d$ will be denoted by $\lambda_d$ whereas the normalized Lebesgue measure on $\TT^d$ will be denoted by $\sigma_d.$
The unit vector $(1,\dots,1)$ will be denoted by $e$.
For a vector $x\in X^d$ and $I=\{i_1,\dots,i_p\}\subset \{1,\dots,d\},$ $x_I$ will mean $(x_{i_1},\dots,x_{i_p})$.

\section{Terminology and statement of the main result}\label{sec:statement}
We start with a function $\varphi\in\mathcal O(\DD^3,\DD)$ which admits a $\mathcal C^3$-smooth extension to $\overline{\DD^3}$. Assume that $\varphi(e)=1$ and let us write
\begin{align*}
\varphi(z)&=1+\frac{\partial\varphi}{\partial z_1}(e)(z_1-1)+\frac{\partial\varphi}{\partial z_2}(z_2-1)+\frac{\partial\varphi}{\partial z_3}(e)(z_3-1)\\
&\quad\quad+\sum_{j+k+l=2}a_{j,k,l}(z_1-1)^j(z_2-1)^k(z_3-1)^l+O(|z_1-1|^3+|z_2-1|^3+|z_3-1|^3).
\end{align*}
By the Julia-Caratheodory theorem (see Lemma \ref{lem:julia} below), $\frac{\partial \varphi}{\partial z_k}(e)\geq 0$ for $k=1,2,3$ and these derivatives cannot simultaneously be equal to $0.$ Hence, if we look at $\varphi(\theta)$, we may write
\begin{align*}
\Imm \varphi(\theta)&=L(\theta)+o(|\theta_1|+|\theta_2|+|\theta_3|) \label{eq:impart} \\
\Ree \varphi(\theta)&=1-Q(\theta)+O\left(|\theta_1|^3+|\theta_2|^3+|\theta_3|^3\right) 
\end{align*}
with $L$ a nonzero linear form and $Q$ a quadratic form. 

We now turn to a map $\phi\in\mathcal O(\DD^3,\DD^3)$ such that $\phi$ extends to a $\mathcal C^3$-smooth function on $\overline{\DD}^3$.
In order to prove or disprove continuity of $C_\phi$ on $H^2(\DD^3),$ it will be important
to deeply study the behaviour of $\phi$ near the points $\xi\in\TT^3$ such that $\phi(\xi)\notin \DD^3.$ This means that there exists $I\subset\{1,2,3\}$, $I\neq\varnothing$ 
such that $\phi_I(\xi)\in \TT^{|I|}$. For $i\in I$, the map $\phi_i$ can be expanded as above, giving rise to a linear map $L_i$ and a quadratic form $Q_i$. The mutual 
properties (and dependencies) of the maps $L_i,$ $Q_i,$ will be the key of our work.

The most difficult case will be when there exists $\xi\in\TT^3$, $I\subset\{1,2\}$ with $|I|=2,$ $\phi_I(\xi)\in\TT^2$,
$\nabla \phi_{i_1}$ and $\nabla\phi_{i_2}$ are linearly dependent. 
Without loss of generality, we may assume $\xi=e$, $I=\{1,2\}$ and $\phi_I(\xi)=(1,1)$. We may write for $j=1,2$,
\begin{eqnarray}
\Imm \phi_j(\theta)&=&\kappa_j L(\theta)+o(|\theta_1|+|\theta_2|+|\theta_3|) \label{eq:impart} \\
\Ree \phi_j(\theta)&=&1-Q_j(\theta)+O\left(|\theta_1|^3+|\theta_2|^3+|\theta_3|^3\right) \label{eq:realpart}
\end{eqnarray}
where $L$ is a nonzero linear form, $\kappa_j$ are nonzero real numbers and $Q_j$ are quadratic forms. 
Since $\Ree \phi_j\leq 1,$ the signature of each $Q_j$ should be $(n_j,0)$ with $n_j\geq 0$.
Therefore there exists $m\in\{0,\dots,3\}$ such that
the signature of $Q_1+Q_2$ is equal to $(m,0)$. We define 
$$s(\phi,I,\xi)=m.$$
Each $Q_j$ may be written $\sum_{k=1}^{n_j}L_{j,k}^2.$ This decomposition is not unique. However, it is easy to check that 
$$\vect(L_{j,k}:\ j=1,2,\ k=1,\dots,n_j)^\perp=\ker(Q_1+Q_2)$$
so that $\dim(\vect(L_{j,k})^\perp)=3-m$. Hence we may pick $(L_1,\dots,L_m)$ 
a basis of $\vect(L_{j,k})$. It should be observed that $L$ belongs to $\vect(L_1,\dots,L_m)$. Otherwise, we could find $\theta=(\theta_1,\theta_2,\theta_3)$ such that 
$L(\theta)\neq 0$ and $L_{j,k}(\theta)=0$ for $j=1,2$ and $k=1,\dots,n_j$. But this would contradict that
$$|\Imm \phi_j(\veps\theta)|^2+|\Ree \phi_j(\veps\theta)|^2\leq 1$$
for small values of $\veps.$ In particular we get $s(\phi,I,\xi)=m\geq 1.$

Now, we complete the family $(L_1,\dots,L_m)$ into $(L_1,\dots,L_3)$ so that this 
last family is a basis of $(\RR^3)^*$. We then get
$$\Imm \phi_j(\theta)=\kappa_j\left(L(\theta)+\sum_{\substack{1\leq k\leq m\\ k\leq l\leq 3}}a_{j,k,l}L_kL_l+\sum_{m+1\leq k\leq l\leq 3}a_{j,k,l}L_kL_l\right)+O\left(|\theta_1|^3+|\theta_2|^3+|\theta_3|^3\right).$$

We then consider the quadratic form 
$$R=\kappa_2\left(\sum_{m+1\leq k\leq l\leq 3}a_{1,k,l}L_kL_l\right)-\kappa_1\left(\sum_{m+1\leq k\leq l\leq 3}a_{2,k,l}L_kL_l\right)$$
and we define 
$$r(\phi,I,\xi)=\textrm{signature}(R).$$
This last quantity does not depend on the way we define the linear forms $L_i$. 
Indeed, let us assume that $m=1$ which is the most difficult case. If we replace $L_i$ by $L'_i$, we may write for $k=1,2,3$
$$L_k=\beta_{1,k} L'_1+\beta_{2,k} L'_2+\beta_{3,k} L'_3$$
where the matrix
$$B=\begin{pmatrix}
\beta_{2,2}&\beta_{2,3}\\
\beta_{3,2}&\beta_{33}
\end{pmatrix}$$
is invertible since $\textrm{span}(L_1)=\textrm{span}(L'_1).$ Writing $R=\sum_{j,k}c_{j,k}L_jL_k$ with a symmetric matrix $C=(c_{j,k})$ and the corresponding $R'=\sum_{j,k}c'_{j,k}L'_jL'_k$ with a symmetric
matrix $C'=(c'_{j,k})$, then $C'=BCB^T,$ which proves the claim.

\smallskip

The integers $s(\phi,I,\xi)$ and $r(\phi,I,\xi)$  we have just introduced will be the keys to understand the behaviour of $\phi$ at a boundary point of order $2$, when the first derivatives are linearly dependent.

\begin{theorem}\label{thm:main}
Let $\phi:\DD^3\to\DD^3$ be holomorphic such that $\phi$ extends to a $\mathcal C^3$-smooth function on $\overline{\DD}^3$. Then $C_\phi$ is bounded on $H^2(\DD^3)$ if and only if $\phi$ satisfies the following properties:
\begin{enumerate}
\item for any $\xi\in\TT^3$ such that $\phi(\xi)\in \TT^3,$ $\mathrm d\phi(\xi)$ is invertible;
\item for any $I\subset\{1,2,3\}$, $I=\{i_1,i_2\},$ such that $\phi_I(\xi)\in\TT^2$, either
\begin{enumerate}
\item $\nabla\phi_{i_1}(\xi)$ and $\nabla\phi_{i_2}(\xi)$ are linearly independent
\item or $\nabla\phi_{i_1}(\xi)$ and $\nabla\phi_{i_2}(\xi)$ are linearly dependent, 
$s(\phi,I,\xi)=3$
\item or $\nabla\phi_{i_1}(\xi)$ and $\nabla\phi_{i_2}(\xi)$ are linearly dependent, 
$s(\phi,I,\xi)=2$ and $r(\phi,I,\xi)=(1,0)$ or $(0,1)$
\item or $\nabla\phi_{i_1}(\xi)$ and $\nabla\phi_{i_2}(\xi)$ are linearly dependent, 
$s(\phi,I,\xi)=1$ and $r(\phi,I,\xi)=(2,0)$ or $(0,2)$.
\end{enumerate}
\end{enumerate}
\end{theorem}

Observe that conditions (b), (c) or (d) imply that, for any $j\in\{1,2,3\}$ and any $k\in\{1,2\},$
$\frac{\partial\phi_{i_k}}{\partial z_j}(\xi)\neq 0.$ Otherwise $\phi_{i_1}$ and $\varphi_2$ would not depend on the same variable.
But setting $m=s(\phi,I,\xi)$ and $(a,b)=r(\phi,I,\xi)$, we would have $m+a+b\leq 2,$ which is not the case here.

Therefore, as a corollary of Theorem \ref{thm:main} and of \cite[Theorem 10]{Ko22}, we get:
\begin{corollary}
 Let $\phi\in\mathcal O(\DD^3,\DD^3)\cap\mathcal C^3(\overline{\DD^3})$. If $C_\phi$ is continuous on $H^2(\DD^3),$
 then it is continuous on $A^2(\DD^3)$.
\end{corollary}

The converse is false, a counterexample being given by $\phi(z)=(z_1z_2,z_3,z_1z_2z_3,0)$ (see \cite{Ko22}).

\section{Preliminaries}

\subsection{How to prove/disprove continuity on $H^2(\DD^d)$}\label{sec:howto}

To prove and disprove continuity on $H^2(\DD^d)$, we shall use Carleson boxes conditions. Since the results of Chang \cite{Ch79}, it is known that Carleson measures of 
the Hardy space of the polydisc are not easy to describe geometrically since we cannot restrict ourselves to estimate the volume of rectangles.
Nevertheless, going through Bergman spaces, we can still get tractable sufficient conditions. We write the results that we shall use for the general polydisc $\DD^d,$ $d\geq 1$.

For any $\eta\in\TT^d$ and any $\ovd=(\delta_1,\dots,\delta_d)\in (0,2]^d,$
we define the Carleson box at $\eta$ with radius $\ovd$ by
$$S(\eta,\ovd)=\big\{(z_1,\dots,z_d)\in\DD^d:\ |z_k-\eta_k|\leq \delta_k,\ 1\leq k\leq d\big\}.$$
We shall use sometimes Carleson windows instead of Carleson boxes. For $\eta=e^{i\theta}\in\TT^d$ and $\ovd\in(0,2]^d$, 
the Carlson window at $\eta$ with radius $\ovd$ is 
$$W(\eta,\ovd)=\big\{z=(r_1e^{it_1},\dots,r_de^{it_d})\in \DD^d:\ 1-\delta_k\leq r_k\leq 1\textrm{ and }|\theta_k-t_k|\leq \delta_k,\ 1\leq k\leq d\big\}.$$
Since there exists $c>0$ such that, for all $\eta\in\TT^d$ and all $\ovd\in(0,2]^d,$
$$S(\eta,\ovd)\subset W(\eta,\ovd)\subset S(\eta,c\ovd),$$
in all results that we will stay, we may replace $S(\eta,\ovd)$ by $W(\eta,\ovd),$ and conversely.

 We use $dA$ to denote the normalized area measure on the unit disk $\mathbb D$ and for $\beta>-1$, we write
$$dA_\beta(z)=(\beta+1)(1-|z|^2)^{\beta}dA(z).$$
We also let 
$$dV_\beta(z)=dA_\beta(z_1)\dots dA_\beta(z_d),$$
the product measure on $\DD^d.$

\begin{lemma}\label{lem:continuityh2}(see \cite{BAYPOLY})
Let $\phi:\DD^d\to\DD^d$ be holomorphic. Assume that there exist $C>0$ and $\veps>0$ such that, for all $\ovd\in (0,\veps]^d,$ for all $\eta\in\TT^d$, for all $\beta\in(-1,0]$, 
$$V_\beta\left(\phi^{-1}(S(\eta,\ovd))\right)\leq C\delta_1^{2+\beta}\cdots\delta_d^{2+\beta}.$$
Then $C_\phi$ is continuous on $H^2(\DD^d)$. 
\end{lemma}

Regarding necessary conditions, we can use the following result, which is the easy part of \cite{Ch79}.

\begin{lemma}\label{lem:discontinuityh2}
Let $\phi:\DD^d\to\DD^d$ be holomorphic and belonging to $\mathcal C(\overline{\DD}^d)$ Assume that $C_\phi$ is continuous on $H^2(\DD^d)$. Then there exists
$C>0$ such that, for all $\eta\in\TT^d$, for all $\ovd\in]0,2]^d,$
\begin{equation}\label{eq:necessarycondition}
 \sigma_d\left(\left\{e^{i\theta}\in\TT^d:\ \phi(\theta)\in S(\eta,\ovd)\right\}\right)\leq C\delta_1\cdots\delta_d.
\end{equation}
\end{lemma}

\subsection{Volume estimates}

As pointed out by Lemma \ref{lem:continuityh2} and \ref{lem:discontinuityh2}, we shall need volume estimates. The forthcoming lemmas summarize what we need.
We start by minorizing the volume of some subsets of $\RR^2.$

\begin{lemma}\label{lem:volumeestimate1}
There exists $C>0$ such that, for all $\delta\in(0,1)$, 
$$\lambda_2\left(\left\{(x,y)\in\mathbb R^2:\ |x|\leq \delta^{1/3},\ |y|\leq \delta^{1/2},\ |xy|\leq \delta\right\}\right)\geq -C\delta\log(\delta).$$
\end{lemma}
\begin{proof}
Let us consider $y\in[\delta^{2/3},\delta^{1/2}]$. Then any $x\in \left[0,\frac{\delta}{y}\right]$ satisfies $|xy|\leq\delta$ and $|x|\leq \delta \cdot \delta^{-2/3}=\delta^{1/3}.$ Therefore, the measure of the set we are looking for is greater than 
$$\int_{\delta^{2/3}}^{\delta^{1/2}}\frac{\delta}ydy=-\frac{1}6 \delta\log(\delta).$$
\end{proof}

\begin{lemma}\label{lem:volumeestimate2}
Let $a,b>0$. There exists $C>0$ such that, for all $\delta\in(0,1),$
$$\lambda_2\left(\left\{(x,y)\in \left[-\delta^{1/3},\delta^{1/3}\right]^2:\ |a x^2-b y^2|\leq \delta\right\}\right)\geq -C \delta\log(\delta).$$
\end{lemma}
\begin{proof}
We may assume $a=b=1$. Setting $x'=x-y$, $y'=x+y$, we get
$$\left|x^2-y^2\right|\leq \delta\iff |x'y'|\leq \delta.$$
Hence the result comes immediately from Lemma \ref{lem:volumeestimate1}.
\end{proof}

We now majorize the measure of some subsets of $\mathbb R^2$.

\begin{lemma}\label{lem:volumeestimate3}
For all $\delta>0$, for all $a\in\mathbb R,$
$$\lambda_2\left(\left\{(x,y)\in\RR^2:\ |x^2+y^2-a|<\delta\right\}\right)\leq 2\pi\delta.$$
\end{lemma}
\begin{proof}
If $a<-\delta$, the set is empty. If $a\in[-\delta,\delta]$, the set is the disc of center $0$ and radius $\sqrt{\delta+a}$, its surface is $\pi(\delta+a)\leq 2\pi\delta$. If $a\geq\delta$, the set is the corona 
$$a-\delta\leq x^2+y^2\leq a+\delta;$$
its surface is $\pi(a+\delta)-\pi(a-\delta)=2\pi\delta.$
\end{proof}

Fubini's lemma will be useful to control the $V_\beta$ volume of some subsets of $\DD^d.$

\begin{lemma}\label{lem:volumeestimate4}
 Let $E$ be a measurable subset of $\RR^d,$ let $\ovd\in(0,1]^3$. There exists $C>0$ such that, for all $\beta\in(-1,0],$
 $$V_\beta\left(\left\{z=\big((1-\rho_k)e^{i\theta_k}\big)_{k=1}^d\in\DD^d:\ \theta\in E,\ 0\leq \rho_k\leq \delta_k\right\}\right)\leq C \delta_1^{1+\beta}\cdots\delta_d^{1+\beta}\lambda_d(E).$$
\end{lemma}
\begin{proof}
 This follows from Fubini's theorem and from the inequality
 $$\int_{1-\delta}^1 (\beta+1)r(1-r^2)^\beta dr\leq C\delta^{1+\beta}$$
 for some constant $C>0$ which does not depend on $\delta\in[0,1]$ and on $\beta\in(-1,0].$
\end{proof}

\subsection{Geometric lemmas}

We will also need two geometric lemmas. The first one is the parametrized Morse lemma (see for instance \cite[Sec.~4.44]{BG92}). 
\begin{lemma}\label{lem:parametrizedmorse}
 Let $f:\mathbb R^d\times\mathbb R^{n-d}\to\mathbb R$ be a $\mathcal C^3$-function such that $f(0)=0$, $\mathrm d_yf(0)=0$ 
	and $\mathrm d^2_yf(0)$ is nonsingular. There exists a neighbourhood $\mathcal U$ of $0$ in $\RR^n,$ a diffeomorphism $\Gamma:\mathcal U\to\mathbb R^n,$ $(x,y)\mapsto (x,\gamma(x,y))$
	and a $\mathcal C^1$-map $h:\mathbb R^d\to\mathbb R$ such that, for any $(x,y)\in\mathcal U,$
	$$f(x,y)=\sum_{j=1}^{n-d}\pm \gamma_j(x,y)^2+h(x).$$
\end{lemma}

The second one is a kind of Julia-Caratheodory theorem in several variables (see in particular \cite[Lemma 2.7]{BAYPOLY}).

\begin{lemma}\label{lem:julia}
Let $\varphi:\overline{\DD}^3\to\overline{\DD}$ which is holomorphic in $\DD^3$ and $\mathcal C^1$ on $\overline{\DD}^3$. Assume that $\varphi(e)=1$. Then for
all $k=1,2,3$, $\frac{\partial \varphi}{\partial z_k}(e)\geq 0.$ Moreover, if 
$\frac{\partial \varphi}{\partial z_k}(e)=0,$ then $\varphi$ does not depend on $z_k$.
\end{lemma}

\subsection{Positive definite quadratic forms}

We will also need a simple lemma on quadratic forms.

\begin{lemma}\label{lem:quadratic}
Let $Q$ be a positive definite quadratic form on $\RR^p$, let $R$ be a quadratic form 
on $\RR^p\times \RR^n$ such that its restriction to $\{0\}\times\RR^n$ is positive definite. Then there exists $\kappa>0$ such that 
$$(x,y)\mapsto \kappa Q(x)+R(x,y)$$
is positive definite. 
\end{lemma}
\begin{proof}
We may assume
\begin{align*}
Q(x)&=\sum_{k=1}^p x_k^2\\
R(x,y)&=\sum_{l=1}^n y_l^2+\sum_{1\leq k, k'\leq p} a_{k,k'}x_kx_{k'}+\sum_{
\substack{1\leq k\leq p \\ 1\leq l\leq n}}b_{k,l}x_ky_l.
\end{align*}
Let $M\geq \sup\{|a_{k,k'}|,|b_{k,l}|\}.$ Let also $\delta>0$.
Then one has
\begin{align*}
\sum_{k,l}|b_{k,l}x_ky_l|&\leq \sum_{k,l}M \frac{|x_k|}\delta(\delta |y_l|)\\
&\leq \frac M2\sum_{k,l}\left(\frac{x_k^2}{\delta^2}+\delta^2 y_l^2\right)\\
&\leq \frac {Mn}{2\delta^2}\sum_k x_k^2+\frac{Mp\delta^2}{2}\sum_l y_l^2.
\end{align*}
A similar argument yields
$$\sum_{k,k'}|a_{k,k'}x_kx_{k'}|\leq Mp\sum_k x_k^2.$$
We then adjust $\delta>0$ so that $Mp\delta^2/2<1$ and we choose
$$\kappa>\frac{Mn}{2\delta^2}+Mp.$$
For any $(x,y)\in \mathbb R^p\times\RR^n \backslash\{(0,0)\},$ $\kappa Q(x)+R(x,y)>0$ so that the quadratic form is positive definite.
\end{proof}

\subsection{Self-maps of the unit disc}

Our next lemmas will help us to provide examples of polynomials which are self maps of the unit disc.

\begin{lemma}\label{lem:eps0}
There exists $\veps_0>0$ such that, for all $\veps\in[-\veps_0,\veps_0]$, for all $\theta\in\RR$, 
$$\left|\left(\frac{1+e^{i\theta}}{2}\right)+i\veps (e^{i\theta}-1)^2\right|\leq 1$$
with equality if and only if $\theta\equiv 0\ [2\pi]$.
\end{lemma}
\begin{proof}
Let 
$$g_\veps(\theta)=\left(\frac{1+e^{i\theta}}{2}\right)+i\veps (e^{i\theta}-1)^2$$
so that 
$$g_\veps(\theta)=1+\frac{i\theta}2-\frac{\theta^2}4+i\theta^3g_1(\theta)+\theta^4 g_2(\theta)+i\veps \theta^2 g_3(\theta)+\veps\theta^3 g_4(\theta)$$
where $g_1,\dots,g_4$ are bounded real-valued functions. Hence,
$$|g_\veps(\theta)|^2=\left(1-\frac{\theta^2}4\right)+\theta^3\left(\veps^2g_5(\theta)+\veps g_6(\theta)+g_7(\theta\right) )$$
with $g_5,g_6,g_7$ bounded real-valued functions.
Therefore there exists $\theta_0>0$ such that, for all $|\theta|\leq \theta_0$, for all $|\veps|<1$, $|g_{\veps}(\theta)|\leq 1$ with equality on $[-\theta_0,\theta_0]$ if and
only if $\theta=0$. Let us now consider $c\in (0,1)$ such that
$$\forall \theta\in [-\pi,\pi]\backslash [-\theta_0,\theta_0], \left|\frac{1+e^{i\theta}}{2}\right|\leq c.$$
Then the choice $\veps_0=\frac{1-c}8$ fulfills the requirements.
\end{proof}

\begin{lemma}
For all $z\in\overline{\DD}$, $|3+6z-z^2|\leq 8$ with equality if and only if $z=1.$
\end{lemma}
\begin{proof}
 Let $z=e^{i\theta}\in\TT$. Then 
 $$|3+6z-z^2|^2\leq 64\iff (3+6\cos(\theta)-\cos(2\theta))^2+(6\sin(\theta)-\sin(2\theta))^2\leq 64.$$
 Developing, reducing and using $\cos(2\theta)=2\cos^2(\theta)-1$, $\sin(2\theta)=2\sin(\theta)\cos(\theta),$
 $\sin^2(\theta)+\cos^2(\theta)=1,$ we easily get that this is equivalent to 
 $$\cos(\theta)(2-\cos(\theta))\leq 1$$
 and this last inequality is verified for all $\theta\in\RR$ with equality if and only if $\theta\equiv 0\ [2\pi].$
\end{proof}

\begin{lemma}\label{lem:eps1}
 There exists $\veps_1>0$ such that, for all $\veps\in[-\veps_1,\veps_1]$, for all $z\in\overline{\DD}$, 
 $$\left|\frac{3+6z-z^2}8+2i\veps(z-1)^2-i\veps (z-1)^3\right|\leq 1$$
 with equality if and only if $z=1.$
\end{lemma}
\begin{proof}
 We set 
 $$G(\theta)=\left|\frac{3+6e^{i\theta}-e^{2i\theta}}8+2i\veps(e^{i\theta}-1)^2-i\veps(e^{i\theta}-1)^3\right|^2$$
 and we expand it around $0.$ We find
 $$G(\theta)=1+\left(4\veps^2-\frac 3{64}\right)\theta^4+\theta^5(G_1(\theta)+\veps G_2(\theta)+\veps^2 G_3(\theta))$$
 with $G_1,G_2,G_3$ bounded real-valued functions. Therefore there exists $\theta_0>0$ such that,
 for all $|\theta|\leq \theta_0$, for all $\veps^2<3/512$, $|G(\theta)|\leq 1$ with equality on $[-\theta_0,\theta_0]$
 if and only if $\theta=0$. Let us consider $c\in (0,1)$ such that 
 $$\forall \theta\in[-\pi,\pi]\backslash [-\theta_0,\theta_0],\ \left|\frac{3+6e^{i\theta}-e^{2i\theta}}{8}\right|\leq c.$$
 Then the choice $\veps_1=\frac{1-c}{32}<\sqrt {3/512}$ fulfills the requirements.
\end{proof}

\section{Discontinuity when the symbol has a boundary point of order of contact $2$}

The key point to prove Theorem \ref{thm:main} is the study of the case of holomorphic maps admitting points with order of contact equal to $2.$
We begin by showing that the conditions given in that theorem are necessary.

\begin{theorem}\label{thm:necconditions}
Let $\phi:\overline{\DD}^3\to\DD^3$ be holomorphic on $\DD^3$ and $\mathcal C^3$-smooth on $\overline{\DD}^3.$
Assume that there exist $I\subset\{1,2,3\}$ with $|I|=2$, $I=\{i_1,i_2\}$ and $\xi\in\TT^3$ such that $\phi_I(\xi)\in\TT^2$. 
Assume that $\nabla \phi_{i_1}(\xi)$ and $\nabla\phi_{i_2}(\xi)$ are linearly dependent, and that 
\begin{itemize}
\item either $s(\phi,I,\xi)=1$ and $r(\phi,I,\xi)\neq (2,0)$ or $(0,2)$
\item or $s(\phi,I,\xi)=2$ and $r(\phi,I,\xi)\neq (1,0)$ or $(0,1)$.
\end{itemize}
Then $C_\phi$ is not continuous on $H^2(\DD^3).$
\end{theorem}

\begin{proof}
We may always assume that $\xi=e,$ that $I=\{1,2\}$ and that $\phi_I(e)=(1,1)$.
We shall prove that condition \eqref{eq:necessarycondition} breaks down.  Because it only depends 
on what happens on the unit polycircle, we are allowed to make a linear change of variables on the coordinates $(\theta_1,\theta_2,\theta_3)$; this will not changed
the volume estimates. Therefore, in what follows, we will always assume that 
the linear map appearing in \eqref{eq:impart} is $L(\theta)=\theta_1$. We shall split the proof into several cases and we will intensively use
the notations of Section \ref{sec:statement}.

\medskip

{\bf Case 1.} $s(\phi,I,\xi)=1$ and $r(\phi,I,\xi)=(0,0)$. Since $L(\theta)=\theta_1$ and $s(\phi,I,\xi)=1,$ we may choose $L_2=\theta_2$ and $L_3=\theta_3$
and write, for $j=1,2,$ keeping only the terms of order $2$ which can be expressed with $\theta_2$ and $\theta_3,$
$$\Imm \phi_j(\theta)=\kappa_j(\theta_1+a_j\theta_2^2+b_j\theta_2\theta_3+c_j\theta_3^2)+O(\theta_1^2+|\theta_1\theta_2|+|\theta_1\theta_3|+|\theta_2|^3+|\theta_3|^3).$$
Since $r(\phi,I,\xi)=(0,0)$, we get $a_1=a_2=:a,$ $b_1=b_2=:b,$ $c_1=c_2=:c.$ Since $s(\phi,I,\xi)=1$, we may also write
$$\Ree\phi_j(\theta)=1-d_j\theta_1^2+O(|\theta_1|^3+|\theta_2|^3+|\theta_3|^3)$$
with $d_j> 0$, for $j=1,2$.
Let $\delta>0$ and set $\bar\delta=(\delta,\delta,2).$
 We choose any $(\theta_2,\theta_3)\in [-\delta^{1/3},\delta^{1/3}]$
and then, any $\theta_1\in [-a\theta_2^2-b\theta_2\theta_3-c\theta_3^2-\delta,-a\theta_2^2-b\theta_2\theta_3-c\theta_3^2+\delta].$ 
The volume of these $\theta=(\theta_1,\theta_2,\theta_3)$ is greater than $\delta^{5/3}\gg \delta^2.$
Moreover, for these $\theta$, observing that $|\theta_1|\lesssim \theta^{2/3}$, it is easy to check that 
\begin{align*}
 |\Imm\phi_j(\theta)|&\lesssim |\theta_1+a\theta_2^2+b\theta_2\theta_3+c\theta_3^2|+O(\delta^{4/3}+\delta)\\
 &\lesssim \delta
\end{align*}
and 
$$ |\Ree\phi_j(\theta)-1|\lesssim\delta.$$
Therefore, 
$$ \sigma_3\left(\left\{e^{i\theta}\in\TT^3:\ \phi(\theta)\in S(\eta,\delta)\right\}\right)\gg \delta^2$$
which prevents $C_\phi$ to be continuous.

\medskip

{\bf Case 2.} $s(\phi,I,\xi)=1$ and $r(\phi,I,\xi)=(1,0)$ or $(0,1)$. Keeping the notations of
Section \ref{sec:statement} and doing a linear change of variables, we may assume that 
$$R(\theta_2,\theta_3)=*\theta_3^2.$$
This means that we may write, for $j=1,2,$
$$\Imm \phi_j(\theta)=\kappa_j(\theta_1+a\theta_2^2+b\theta_2\theta_3+c_j\theta_3^2)+O\left(\theta_1^2+|\theta_1\theta_2|+|\theta_1\theta_3|+|\theta_2|^3+|\theta_3|^3\right)$$
and we always have
$$\Ree\phi_j(\theta)=1-d_j\theta_1^2+O\left(|\theta_1|^3+|\theta_2|^3+|\theta_3|^3\right).$$

We first assume that $a\neq 0$ and, without loss of generality, that $a=1$. Writing
$$\theta_2^2+b\theta_2\theta_3=\left(\theta_2+\frac b2\theta_3\right)^2-\frac{b^2}4\theta_3^2$$
and setting $\theta'_2=\theta_2+\frac b2\theta_3,$ $\theta'_3=\theta_3$, we may also assume that $b=0.$
Let $\delta>0$ and consider any $(\theta_1,\theta_3)\in [-2\delta^{2/3},-\delta^{2/3}]\times [-\delta^{1/2},\delta^{1/2}].$ 
For these couples $(\theta_1,\theta_3)$, consider any $\theta_2$ in the interval 
$$I_{\theta_1}=\left[\sqrt{-\theta_1-\delta},\sqrt{-\theta_1+\delta}\right].$$
The length of $I_{\theta_1}$ satisfies
$$|I_{\theta_1}|\gtrsim \frac{\delta}{\delta^{1/3}}=\delta^{2/3}$$
so that 
$$
\sigma_3\left(\left\{e^{i\theta}\in\TT^3:\ (\theta_1,\theta_3)\in [-2\delta^{2/3},\delta^{2/3}]\times [-\delta^{1/2},\delta^{1/2}],\ \theta_2\in I_{\theta_1}\right\}\right)\gtrsim \delta^{\frac 43+\frac 12}\gg \delta^2.
$$
Moreover, for these $\theta$, observing that $|\theta_2|\lesssim \delta^{1/3}$, it is easy to check that, for $j=1,2$,
\begin{align*}
|\Imm \phi_j(\theta)|&\lesssim \left|\theta_1+\theta_2^2\right|+O\left(|\theta_1|^2+|\theta_1\theta_2|+|\theta_1\theta_3|+|\theta_2|^3+|\theta_3|^2\right)\\
&\lesssim \delta
\end{align*}
 and $|\Ree \phi_j(\theta)-1|\lesssim \delta$.
 Therefore $C_\phi$ cannot be continuous. We now handle the case $a=0$. This means that, for $j=1,2,$
 \begin{align*}
 \Imm \phi_j(\theta)&=\kappa_j\left(\theta_1+b\theta_2\theta_3+c_j\theta_3^2\right)+O\left(\theta_1^2+|\theta_1\theta_2|+|\theta_1\theta_3|+|\theta_2|^3+|\theta_3|^3\right)\\
 \Ree \phi_j(\theta)&=1-d_j \theta_1^2+O\left(|\theta_1|^3+|\theta_2|^3+|\theta_3|^3\right).
 \end{align*}
 For $\delta>0,$ we set $K_\delta=\left\{(\theta_2,\theta_3):\ |\theta_2|\leq \delta^{1/3},\ |\theta_3|\leq \delta^{1/2},\ |\theta_2\theta_3|\leq\delta\right\}$
 which satisfies by Lemma \ref{lem:volumeestimate1} $\lambda_2(K_\delta)\gtrsim -\delta\log(\delta).$ Then any $\theta\in [-\delta,\delta]\times K_\delta$ satisfies 
 $$|\Imm \phi_j(\theta)|\lesssim \delta\textrm{ and }|\Ree \phi_j(\theta)-1|\lesssim\delta.$$
 Since $\sigma_3\left(\left\{e^{i\theta}:\ \theta\in [-\delta,\delta]\times K_\delta\right\}\right)\gtrsim -\delta^2\log(\delta)$, $C_\phi$ cannot be continuous.
 
 \medskip
 
 {\bf Case 3.} $s(\phi,I,\xi)=1$ and $r(\phi,I,\xi)=(1,1)$. Doing
  a linear change of variables, we may and shall assume that 
 $$R(\theta_2,\theta_3)=\kappa_1\kappa_2(\theta_2^2-\theta_3^2).$$
 This means that we may write, for $j=1,2,$
 $$\Imm \phi_j(\theta)=\kappa_j\left(\theta_1+a_j \theta_2^2+b\theta_2\theta_3+c_j\theta_3^2\right)+O\left(\theta_1^2+|\theta_1\theta_2|+|\theta_1\theta_3|+|\theta_2|^3+|\theta_3|^3\right)$$
 with $a_1=a_2+1,$ $c_1=c_2-1$.  For $\delta\in (0,1)$, we set 
 $$E(\delta)=\left\{(\theta_2,\theta_3)\in \left[-\delta^{1/3},\delta^{1/3}\right],\ \left|\theta_2^2-\theta_3^2\right|\leq\delta\right\}.$$
 For $(\theta_2,\theta_3)\in E(\delta),$ let 
 $$I_{\theta_2,\theta_3}(\delta)=\left\{\theta_1:\ \left|\theta_1+a_1\theta_2^2+b\theta_2\theta_3+c_1\theta_3\right|^2\leq\delta\right\}.$$
 Then Lemma \ref{lem:volumeestimate2} implies that
 $$\lambda_2\left(\left\{\theta:\ (\theta_2,\theta_3)\in E(\delta),\ \theta_1\in I_{\theta_2,\theta_3}(\delta)\right\}\right)\gtrsim -\delta^2\log(\delta).$$
 On the other hand, provided $(\theta_2,\theta_3)\in E(\delta)$ and $\theta_1\in I_{\theta_2,\theta_3}(\delta)$, we easily verify that $\left|\theta_1\right|\lesssim \delta^{2/3}$
which in turn implies 
\begin{align*}
\left|\Imm \phi_1(\theta)\right|&\lesssim \delta+O\left(|\theta_1\theta_2|+|\theta_1\theta_3|+|\theta_1|^2+|\theta_2|^3+|\theta_3|^3\right)\\
&\lesssim \delta,
\end{align*}
\begin{align*}
\left|\Imm \phi_2(\theta)\right|&\lesssim \delta+\left|\theta_2^2-\theta_3^2\right|+O\left(|\theta_1\theta_2|+|\theta_1\theta_3|+|\theta_1|^2+|\theta_2|^3+|\theta_3|^3\right)\\
&\lesssim \delta,\\
\left|\Ree \phi_j(\theta)-1\right|&\lesssim \delta,\ j=1,2.
\end{align*}
Therefore, $ \sigma_3\left(\left\{e^{i\theta}\in\TT^3:\ \phi(\theta)\in S(\eta,\bar\delta)\right\}\right)\gg\delta^2$ for $\bar\delta=(\delta,\delta,\delta)$, 
which prevents $C_\phi$ to be continuous.

\medskip

{\bf Case 4.} $s(\phi,I,\xi)=2$ and $r(\phi,I,\xi)=(0,0)$. Using the notations of Section \ref{sec:statement} and doing a linear change of variables, we may assume that $L(\theta)=\theta_1$
and that the linear forms $L_{j,k}$, $1\leq j,k\leq 2,$ are linear combinations of $\theta\mapsto\theta_1$
and $\theta\mapsto\theta_2$. Therefore, we may write 
\begin{align*}
\Imm \phi_j(\theta)&=\kappa_j\left(\theta_1+a_j\theta_1\theta_3+b_j\theta_2\theta_3
+c\theta_3^2\right)+O\left(|\theta_1|^2+|\theta_2|^2+|\theta_3|^3\right)\\
\Ree \phi_j(\theta)&=1+O\left(|\theta_1|^2+|\theta_2|^2+|\theta_3|^3\right).
\end{align*}
Let $\delta>0$ and 
$$E(\delta)=\left\{(\theta_2,\theta_3):\ |\theta_2|\leq\delta^{1/2},\ |\theta_3|\leq \delta^{1/3},\ |\theta_2\theta_3|\leq\delta\right\}.$$
By Lemma \ref{lem:volumeestimate1}, $\lambda_2(E(\delta))\gtrsim-\delta\log(\delta).$
For $\theta_3\in\mathbb R$, we set 
$$I_{\theta_3}(\delta)=\left\{\theta_1:\ \left|\theta_1+c\theta_3^2\right|\leq\delta\right\}$$
so that 
$$\sigma_3\left(\left\{e^{i\theta}:\ (\theta_2,\theta_3)\in E(\delta),\ \theta_1\in I_{\theta_3}(\delta)\right\}\right) \gtrsim -\delta^2\log(\delta).$$
Moreover, for $\theta$ such that $(\theta_2,\theta_3)\in E(\delta)$ and $\theta_1\in I_{\theta_3}(\delta),$ one has $|\theta_1|\lesssim \delta^{2/3},$ 
$|\theta_2|\leq \delta^{1/2}$ and $|\delta_3|\leq \delta^{1/3}$ so that for $j=1,2$, $|\Ree \phi_j(\theta)-1|\lesssim\delta$ and 
$\left|\Imm\phi_j(\theta)\right|\lesssim\delta.$ Again, $C_\phi$ cannot be continuous.
\end{proof}


\section{Volume estimates when the symbol has a boundary point of order 2} \label{sec:volumeestimate}
The aim of this section is to prove the following result, which is the main step towards the proof of the positive part of Theorem \ref{thm:main}.

\begin{theorem}\label{thm:volumelinearlydependent}
Let $\phi\in\mathcal O(\DD^3,\DD^3)\cap\mathcal C^3(\overline{\DD^3})$.
Let $\xi\in\TT^3$, $I\subset\{1,2\}$ with $|I|=2$, $I=\{i_1,i_2\}$, $\phi_I(\xi)\in\TT^2$ and $\nabla\phi_{i_1}(\xi),\nabla\phi_{i_2}(\xi)$ are linearly dependent. 
There exist a constant $C>0$ and a neighbourhood $\mathcal U$ of $\xi$ such that,
for all $\beta\in(-1,0],$ for all $\eta\in\TT^2$, for all $\ovd\in (0,+\infty)^2,$
$$V_\beta\left(\left\{(z_1,z_2,z_3)\in\mathcal U\cap \DD^3:\ |\phi_{i_j}(z)-\eta_j|\leq \delta_j,\ j=1,2\right\}\right)\leq C\delta_1^{2+\beta}\delta_2^{2+\beta}$$
in the following cases:
\begin{itemize}
\item $s(\phi,I,\xi)=3$;
\item $s(\phi,I,\xi)=2$ and $r(\phi,I,\xi)\in \{(1,0);(0,1)\}$;
\item $s(\phi,I,\xi)=1$ and $r(\phi,I,\xi)\in \{(2,0);(0,2)\}.$
\end{itemize}
\end{theorem}
\begin{proof}
To fix the notations, we assume that $I=\{1,2\}$, that $\xi=e=(1,1,1)$ and that $\phi_I(e)=(1,1)$. Without loss of generality we assume $\delta_1\leq\delta_2.$ We start dealing with the case $s(\phi,I,\xi)=3$ and we first prove an intermediate result:

\medskip

{\bf Fact.} There exist $C>0$ and a neighbourhood $\mathcal O$ of $0$ in $\mathbb R^3$ such that, for all $\ovd\in(0,+\infty)^2,$ for all $\eta\in\TT^2,$
\begin{equation}\label{eq:volumelinearlydependent1}
\lambda_3\left(\left\{\theta\in\mathcal O:\ |\phi_j(\theta)-\eta_j|\leq \delta_j,\ j=1,2\right\}\right)\leq C\delta_1\delta_2.
\end{equation}

\bigskip

We fix $\eta\in\TT^2$ and $\ovd\in(0,+\infty)^2$.  We know that, for $j=1,2$, 
$$\Imm \phi_j(\theta)=\kappa_j\left(L(\theta)+G_j(\theta)\right)$$
with $G_j(0)=0,$ $\nabla G_j(0)=0.$ Doing a linear change of variables, which does not affect neither the desired inequality \eqref{eq:volumelinearlydependent1}
nor the signature of $Q$, we may assume that $L(\theta)=\theta_1$.
We set $U(\theta)=(\theta_1+G_1(\theta),\theta_2,\theta_3).$ There exist two
neighbourhoods $\mathcal O_1$ and $\mathcal O_2$ of $0$ and $C_1>0$ such that
$U$ is a diffeomorphism from $\mathcal O_1$ onto $\mathcal O_2$ and 
$$\left\{
\begin{array}{rcll}
|\det(\mathrm dU)|&\leq& C_1&\textrm{on }\mathcal O_1\\
|\det(\mathrm dU^{-1})|&\leq& C_1&\textrm{on }\mathcal O_2.
\end{array}\right.$$
We then write
$$\Ree \phi_j(\theta)=1-Q_j(\theta)+H_j(\theta)$$
with $H_j(\theta)=O\left(|\theta_1|^3+|\theta_2|^3+|\theta_3|^3\right).$
Because $\mathrm dU(0)=I_3$, for $u\in \mathcal O_2,$
$$\Ree \phi_j\circ U^{-1}(u)=1-Q_j(u)+K_j(u)$$
with $K_j(u)=O\left(|u_1|^3+|u_2|^3+|u_3|^3\right).$ Let us now define
$$E(\ovd)=\left\{\theta\in\RR^3:\ |\phi_j(\theta)-\eta_j|\leq\delta_j,\ j=1,2\right\}$$
and let us write $\eta_j=x_j+iy_j.$ Provided $u\in U(E(\ovd)\cap\mathcal O_1),$
\begin{eqnarray}
\nonumber&&|-Q_1(u)-Q_2(u)+K_1(u)+K_2(u)+2-x_1-x_2|\\
&&\quad\quad\quad\quad\leq|\Ree \phi_1\circ U^{-1}(u)-x_1|+|\Ree \phi_2\circ U^{-1}(u)-x_2| \nonumber \\
&&\quad\quad\quad\quad\leq2\delta_2. \label{eq:volumelinearlydependent2}
\end{eqnarray}
We set $F(u)=-Q_1(u)-Q_2(u)+K_1(u)+K_2(u)$. Since the signature of $Q_1+Q_2$ is $(3,0)$, the matrix 
$$\left(\frac{\partial^2 F}{\partial u_i\partial u_k}(0)\right)_{2\leq i,j\leq 3}$$
is negative definite. Hence we may apply the parametrized Morse lemma (Lemma \ref{lem:parametrizedmorse}): there exist two neighbourhoods $\mathcal O_2'\subset\mathcal O_2$ 
and $\mathcal O_3$ of $0$, a constant $C_2>0$ and a diffeomorphism $V:\mathcal O'_2\to\mathcal O_3,$ $V(u_1,u_2,u_3)=(u_1,V_2(u_1,u_2,u_3),V_3(u_1,u_2,u_3))$ and a $\mathcal C^1$-map
$h:\mathbb R\to\mathbb R$ such that, for all $v\in\mathcal O_3,$ 
$F\circ V^{-1}(v)=-v_2^2-v_3^2+h(v_1)$ and 
$$\left\{
\begin{array}{rcll}
|\det(\mathrm dV)|&\leq C_2&\textrm{on }\mathcal O'_2\\
|\det(\mathrm dV^{-1})|&\leq C_2&\textrm{on }\mathcal O_3.
\end{array}\right.$$
We finally set $\mathcal O=U^{-1}(\mathcal O'_2)\subset\mathcal O_1.$ Therefore,
provided $v\in V\circ U(E(\ovd)\cap\mathcal O),$ we infer from $|\Imm \phi_1-\eta_1|\leq\delta_1$ on $E(\delta)$ and from \eqref{eq:volumelinearlydependent2} that
$$\left\{
\begin{array}{rcl}
|v_1-y_1|&\leq&\delta_1/\kappa_1\\
|-v_2^2-v_3^2+h(v_1)+2-x_1-x_2|&\leq&2\delta_2
\end{array}\right.$$
It then follows from Fubini's theorem and Lemma \ref{lem:volumeestimate3} that 
$$\lambda_3\big(V\circ U(E(\ovd)\cap\mathcal O)\big)\lesssim\delta_1\delta_2$$
where the involved constant does not depend on $\eta_1,\eta_2$ and $\delta$. 
We then 
get the fact by the change of variables formula.

\bigskip

We now prove the statement of the theorem (still in the case $s(\phi,I,\xi)=3$). We
write each $z_k=(1-\rho_k)e^{i\theta_k}.$ A Taylor expansion of $\phi_j$ near $e$ reveals that
$$|\phi_j(z)|^2=1+\sum_{k=1}^3 \rho_k F_{j,k}(\rho,\theta)+G_j(\theta)$$
where $F_{j,k}(0,0)=-2\frac{\partial \phi_j}{\partial z_k}(e).$ 
Observe that $\frac{\partial \phi_j}{\partial z_k}(e)\neq 0$, otherwise 
we would get $\frac{\partial \phi_1}{\partial z_k}(e)=\frac{\partial\phi_2}{\partial z_j}(e)=0$ (recall that $\nabla\phi_1(e)$ and $\nabla\phi_2(e)$ are linearly dependent). This would imply that neither $\phi_1$ nor $\phi_2$ would depend on $z_k$ (see Lemma \ref{lem:julia}), which contradicts 
$s(\phi,I,\xi)=3.$

Choosing $\rho_k=0$, we get $G_j(\theta)\leq 0$ for all $\theta$. On the other hand, provided $|\phi_1(z)-\eta_1|<\delta_1,$ which implies $|\phi_1(z)|^2\geq 1-2\delta_1,$
we necessarily get 
$$\sum_{k=1}^3 \rho_k F_{1,k}(\rho,\theta)\geq -2\delta_1.$$
Since $F_{1,k}(0,0)<0$, we get the existence of a neighbourhood $\mathcal U_1$ of $e$ in $\overline{\DD^3}$ 
such that, for all $z\in\mathcal U_1\cap \phi_I^{-1}(S(\eta,\ovd))$, $\rho_k\lesssim\delta_1$, $k=1,2,3$ where the implied constant does not depend on $\ovd$. Shrinking $\mathcal U_1$ if necessary, we may assume that it is convex. 
Now, observe that for $z\in \mathcal U_1\cap \phi_I^{-1}(S(\eta,\ovd)),$ writing always $z_k=(1-\rho_k)e^{i\theta_k},$ for $j=1,2,$
\begin{align*}
|\phi_j(z)-\phi_j(\theta)|&\leq \sup_{\mathcal U_1} \|\nabla\phi_j\|\cdot \max_k \|z_k-e^{i\theta_k}\|\\
&\lesssim \delta_1\leq\delta_j.
\end{align*}
Hence, there exists $D>0$ such that, for all $z\in\mathcal U_1\cap \phi_I^{-1}(S(\eta,\ovd))$, we have $e^{i\theta}\in\phi_I^{-1}(S(\eta,D\ovd)).$
Let $\mathcal O$ be given by the fact (for $D\ovd$) and let us set
$$\mathcal U=\mathcal U_1\cap\left\{\big((1-\rho_k)e^{ik\theta}\big)_{k=1}^3:\ \theta\in\mathcal O,\ 0\leq \rho_k\leq \delta_1\right\}.$$
We get from Lemma \ref{lem:volumeestimate4} the existence of $C>0$ such that, for all $\beta\in (-1,0]$, for all $\eta\in\TT^2$, for all $\ovd\in(0,2]^2,$
\begin{align*}
V_\beta\left(\mathcal U\cap \phi_I^{-1}(S(\eta,\ovd))\right)&\leq C\delta_1^{3(1+\beta)}\delta_1\delta_2\\
&\leq C\delta_1^{2+\beta}\delta_2^{2+\beta}.
\end{align*}

\medskip

We turn to the case $s(\phi,I,\xi)=2$ and $r(\phi,I,\xi)=(1,0)$ or $(0,1)$. Arguing as in the case $s(\phi,I,\xi)=3,$ we just need to prove \eqref{eq:volumelinearlydependent1}.
Indeed, the last part of the proof only depends on the property that neither $\phi_1$ nor $\phi_2$ do not depend on some coordinate $z_k$ and the assumptions
$s(\phi,I,\xi)=2$ and $r(\phi,I,\xi)=(1,0)$ or $(0,1)$ also prevents this.
Using the notations of Section \ref{sec:statement} and doing a linear change of variables, we may always assume that $L(\theta)=\theta_1$ and that each $L_{j,k}$ is a linear combination of $\theta\mapsto \theta_1$ and $\theta\mapsto\theta_2$. In particular, we may write
$$\Imm \phi_j(\theta)=\kappa_j\left(\theta_1+\sum_{1\leq k\leq l\leq 3}a_{j,k,l}\theta_k\theta_l+\tilde G_j(\theta)\right)$$
with $\tilde G_j(\theta)=O\left(|\theta_1|^3+|\theta_2|^3+|\theta_3|^3\right).$
The condition $r(\phi,I,\xi)=(1,0)$ or $(0,1)$ just means that $a_{2,3,3}-a_{1,3,3}\neq 0$. We first do the same change of variables as in the case $s(\phi,I,\xi)=3$: $U:\mathcal O_1\to\mathcal O_2$ is defined by
$$U(\theta_1,\theta_2,\theta_3)=\left(\theta_1+\sum_{1\leq k\leq l\leq 3}a_{1,k,l}\theta_k\theta_l+\tilde G_1(\theta),\theta_2,\theta_3\right)$$
and satisfies 
$$\left\{
\begin{array}{ll}
|\det(\mathrm dU)|\leq C&\textrm{on }\mathcal O_1\\
|\det(\mathrm dU^{-1})|\leq C&\textrm{on }\mathcal O_2.
\end{array}\right.$$
Therefore, 
\begin{align*}
\Imm \phi_1\circ U^{-1}(u)&=\kappa_1 u_1\\
\Imm \phi_2\circ U^{-1}(u)&=\kappa_2\left(u_1+\sum_{1\leq k\leq l\leq 3}b_{k,l}u_ku_l+H(u)\right)
\end{align*}
with $b_{k,l}=a_{2,k,l}-a_{1,k,l}$ and $H(u)=O\left(|u_1|^3+|u_2|^3+|u_3|^3\right).$ In particular, $b_{3,3}\neq 0$. Without loss of generality, we will assume that $b_{3,3}>0$. We may also write 
$$\Ree \phi_j\circ U^{-1}(u)=1-Q_j(u)+K_j(u)$$
where $Q_j$ is still the quadratic form introduced in Section \ref{sec:statement} (the change of variables
does not affect it since $\mathrm{d}U(0)=I_3$) and $K_j(u)=O\left(|u_1|^3+|u_2|^3+|u_3|^3\right)$. In particular, $Q_1+Q_2$ only depends on $(u_1,u_2)$ and has signature $(2,0)$. By Lemma \ref{lem:quadratic}, there exists $\kappa>0$ so that 
$$R:u\mapsto \kappa(Q_1+Q_2)(u)+\kappa_1\kappa_2\sum_{1\leq k\leq l\leq 3}b_{k,l}u_ku_l$$
is positive definite. On the other hand, writing $\eta_j=x_j+iy_j$, for all $u\in U(E(\ovd)\cap\mathcal O_1)$,
\begin{eqnarray*}
&&\big|R(u)-\kappa(K_1+K_2)(u)+\kappa(-2+x_1+x_2)+\kappa_2 H(u)+\kappa_1 y_2-\kappa_2 y_1\big|\\
&&\quad\quad\quad\quad \leq \kappa |\Ree \phi_1\circ U^{-1}(u)-x_1|+\kappa |\Ree \phi_2\circ U^{-1}(u)-x_2|+\\
&&\quad\quad\quad\quad\quad\quad\kappa_1|\Imm \phi_1\circ U^{-1}(u)-y_1|+\kappa_2|\Imm \phi_2\circ U^{-1}(u)-y_2|\\
&&\quad\quad\quad\quad\leq (2\kappa+\kappa_1+\kappa_2)\delta_2.
\end{eqnarray*}
We now argue exactly as in the proof of the case $s(\phi,I,\xi)=3,$ by considering 
$$F(u)=R(u)-\kappa(K_1+K_2)(u)+\kappa_1\kappa_2 H(u)$$
and by applying the parametrized Morse lemma to it. Details are left to the reader.

\medskip

We end the proof of Theorem \ref{thm:volumelinearlydependent} by solving the case $s(\phi,I,\xi)=1$ and $r(\phi,I,\xi)=(2,0)$ (or $(0,2)$).
We use again the notations of Section \ref{sec:statement} and we assume that $L(\theta)=\theta_1.$ We know that
$$\kappa_2\Imm \phi_1(\theta)-\kappa_1 \Imm \phi_2(\theta)=a\theta_1^2+b\theta_1\theta_2+c\theta_1\theta_3+Q(\theta_2,\theta_3)+H(\theta)$$
where $Q$ is a quadratic form with signature $(2,0)$ (the case $(0,2)$ is symmetric) and $H(\theta)=O\left(|\theta_1|^3+|\theta_2|^3+|\theta_3|^3\right)$. We also write
$$\Imm \phi_1(\theta)=\kappa_1\big(\theta_1+G_1(\theta)\big)$$
with $G_1(\theta)=O\left(|\theta_1|^2+|\theta_2|^2+|\theta_3|^2\right).$ As above, we consider the diffeomorphism $U:\mathcal O_1\to\mathcal O_2$
$$U(\theta)=(\theta_1+G_1(\theta),\theta_2,\theta_3).$$
Since $\mathrm{d}U(0)=I_3$, we still have (with the same coefficients)
$$(\kappa_2\Imm \phi_1-\kappa_1 \Imm \phi_2)\circ U^{-1}(u)=au_1^2+bu_1u_2+cu_1u_3+Q(u_2,u_3)+\tilde H(u)$$
with $\tilde H(u)=O\left(|u_1|^3+|u_2|^3+|u_3|^3\right).$
We now apply the parametrized Morse lemma to $\kappa_2\Imm\phi_1-\kappa_1\Imm\phi_2$ to get a diffeomorphism $V:\mathcal O_2'\subset\mathcal O_2\to \mathcal O_3,$ $V(u_1,u_2,u_3)=(u_1,v_2(u_1,u_2,u_3),v_3(u_1,u_2,u_3))$
and a $\mathcal C^1$-map $h$ such that 
$$(\kappa_2\Imm \phi_1-\kappa_1\Imm\phi_2)\circ(U\circ V)^{-1}(v)=v_2^2+v_3^2+h(v_1).$$
Setting $\mathcal O=U^{-1}(\mathcal O'_2)$, we conclude as before since, for all $\theta\in \mathcal O$ such that $|\phi_j(\theta)-\eta_j|\leq\delta_j,$ $j=1,2,$ for $v=V\circ U(\theta),$ 
$$\left\{
\begin{array}{rcl}
|\kappa_1 v_1-x|&\leq&\delta_1\\
|v_2^2+v_3^2+h(v_1)-y|&\lesssim&\delta_2
\end{array}\right.
$$
for some real numbers $x$ and $y$.
\end{proof}

To cover the other cases of Theorem \ref{thm:main}, we shall need the following result.

\begin{theorem}\label{thm:kosinski}
 Let $\phi\in\mathcal O(\DD^d,\DD^d)\cap\mathcal C^3(\overline{\DD^3})$. Let $\xi\in\TT^d$, let $I=\{i_1,\dots,i_p\}\subset\{1,\dots,d\},$ $I\neq\varnothing$
 such that $\phi_I(\xi)\in\TT^{|I|}$ and $\nabla \phi_{i_1}(\xi),\dots,\nabla\phi_{i_p}(\xi)$ are linearly independent.
 There exist $C>0$ and a neighbourhood $\mathcal U$ of $\xi$ such that, for all $\beta\in(-1,0],$ 
 for all $\eta\in\TT^{|I|},$ for all $\ovd\in(0,+\infty)^{|I|},$
 $$V_\beta\left(\left\{z\in\mathcal U\cap \DD^d:\ |\phi_{i_j}(z)-\eta_j|\leq\delta_j,\ j=1,\dots,p\right)\right\}\leq C\delta_1^{2+\beta}\cdots \delta_p^{2+\beta}.$$
\end{theorem}
The proof of this theorem is contained in \cite{Ko22}: see Lemma 12 and the proof of it.

\section{Proof of our main theorem}

This section is devoted to the proof of Theorem \ref{thm:main}. Let us first prove the sufficient part. We introduce 
$$A_1=\left\{\xi\in\overline{\DD}^3:\ \phi(\xi)\in\TT^3\right\}.$$
Let $\xi\in A_1.$ We first observe that $\xi$ belongs to $\TT^3$. Otherwise if we assumed that $\xi_1\in\DD$, we would obtain that $\phi$ does not depend on $z_1$ which 
would prevent $\mathrm d\phi(\xi)$ to be invertible. By Theorem \ref{thm:kosinski}, there exist
a neighbourhood $\mathcal U(\xi)$ of $\xi$ and $C(\xi)>0$ such that, for all $\eta\in\TT^3,$ for all $\beta\in(-1,0]$, for all $\ovd\in (0,2]^3$,
$$V_\beta\left(\phi^{-1}(S(\eta,\ovd))\cap\mathcal U(\xi)\right)\leq C(\xi)\delta_1^{2+\beta}\delta_2^{2+\beta}\delta_3^{2+\beta}.$$
Since $A_1$ is compact, it can be covered by a finite number of sets $\mathcal U(\xi_{1,1}),\dots,\mathcal U(\xi_{1,p})$. We also get the existence of $\veps_1\in (0,1)$ such that, for all $z\in B_1:=\overline{\DD}^3\backslash\bigcup_k \mathcal U(\xi_{1,k})$, $|\phi_j(z)|<1-2\veps_1$ for some $j\in\{1,2,3\}.$

We then set 
$$A_2=\left\{\xi \in B_1:\ \textrm{there exists }I\subset\{1,2,3\},\ |I|=2,\ \phi_I(\xi)\in\TT^2\right\}.$$
Let $\xi\in A_2$. We first assume that $\nabla \phi_{i_1}(\xi)$ and $\nabla \phi_{i_2}(\xi)$ are linearly independent. Let $j\in\{1,2,3\}\backslash I$ and observe 
that $|\phi_j(\xi)|<1-2\veps_1.$ Then we may again apply Theorem \ref{thm:kosinski} to
get a neighbourhood $\mathcal U(\xi)$ of $\xi$ and $D(\xi)>0$ such that, for all $\eta\in\TT^3,$ for all $\beta\in(-1,0]$, for all $\ovd\in(0,2]^3$, for all $z\in \mathcal U(\xi)$, 
$$\left\{
\begin{array}{l}
V_\beta\left(\phi_I^{-1}(S(\eta_I,\ovd_I))\cap\mathcal U(\xi)\right)\leq D(\xi)\delta_{i_1}^{2+\beta}\delta_{i_2}^{2+\beta}\\
|\phi_j(z)|<1-\veps_1.
\end{array}\right.$$
Since $\phi^{-1}(S(\eta,\ovd))\cap \mathcal U(\xi)$ is empty provided $\delta_j< \veps_1$, we get 
\begin{align*}
V_\beta\left(\phi^{-1}(S(\eta,\ovd))\cap\mathcal U(\xi)\right)&\leq\frac{D(\xi)}{\veps_1^{2+\beta}}\delta_1^{2+\beta}\delta_2^{2+\beta}\delta_3^{2+\beta}\\
&\leq C(\xi)\delta_1^{2+\beta}\delta_2^{2+\beta}\delta_3^{2+\beta}
\end{align*}
where we have set $C(\xi)=D(\xi)/\veps_1^2$ (which does not depend on $\eta$, $\beta\in(-1,0]$ and $\ovd$). 

Assume now that $\nabla \phi_{i_1}$ and $\nabla\phi_{i_2}$ are not linearly independent. We claim that $\xi\in\TT^3$. Otherwise, if $\xi_1\in\DD$, $\phi_{i_1}$
and $\phi_{i_2}$ would not depend on $z_1$. Now, let $\xi'=(1,\xi_2,\xi_3)$. Then it
would be impossible that $\phi$ satisfies the assumptions of the theorem at $\xi'$, since
$\nabla \phi_{i_1}(\xi')$ and $\nabla \phi_{i_2}(\xi')$ are linearly dependent, and
since $\phi_{i_1}$ and $\phi_{i_2}$ do not depend on $z_1$, conditions (2), (b), (c) and (d) cannot happen. We now argue as before, but using Theorem \ref{thm:volumelinearlydependent} to get a neighbourhood $\mathcal U(\xi)$ of $\xi$ and $C(\xi)>0$ such that, for all $\eta\in\TT^3,$ for all $\beta\in(-1,0]$, for all $\ovd\in(0,2]^3$, 
$$V_\beta\left(\phi^{-1}(S(\eta,\ovd))\cap\mathcal U(\xi)\right)\leq C(\xi)\delta_1^{2+\beta}\delta_2^{2+\beta}\delta_3^{2+\beta}.$$
We then cover $A_2$ by a finite number of these sets $\mathcal U(\xi_{2,1}),\dots,\mathcal U(\xi_{2,q})$. We define $\veps_2\in(0,\veps_1)$ such that, for all $z\in B_2:=B_1\backslash \bigcup_k \mathcal U(\xi_{2,k})$, $|\phi_{j_1}(z)|<1-2\veps_2$ and $|\phi_{j_2}(z)|<1-2\veps_2$ for at least two different integers $j_1,j_2$ in $\{1,2,3\}.$

We finally set 
$$A_3=\left\{\xi\in B_2:\ \textrm{there exists }j\in\{1,2,3\},\ \phi_j(z)\in\TT\right\}.$$
Arguing as before, using Theorem \ref{thm:kosinski}, we get a finite covering 
$\mathcal U(\xi_{3,1}),\dots,\mathcal U(\xi_{3,r})$ of $A_3$ and constants $C(\xi_{3,1}),\dots,C(\xi_{3,r})>0$ such that, for all $\eta\in\TT^3$, for all $\beta\in(-1,0]$, for all $\ovd\in(0,2]^3,$
$$V_\beta\left(\phi^{-1}(S(\eta,\ovd))\cap\mathcal U(\xi_{3,k})\right)\leq C(\xi_{3,k})\delta_1^{2+\beta}\delta_2^{2+\beta}\delta_3^{2+\beta}.$$
There also exists $\veps_3\in(0,\veps_2)$ such that, provided $z\in B_3:=A_3\backslash \bigcup_k \mathcal U(\xi_{3,k})$, $|\phi_j(z)|<1-\veps_3$ for all $j=1,2,3$. 

Let now $\ovd\in(0,2]^3$ and assume that at least one of $\delta_1,\delta_2,\delta_3$ is less than $\veps_3$. Then 
$$\phi^{-1}\big(S(\eta,\ovd)\big)=\left(\phi^{-1}(S(\eta,\ovd))\cap \bigcup_{j,k}\mathcal U(\xi_{j,k})\right)\cup \left(\phi^{-1}(S(\eta,\ovd))\cap B_3\right).$$
Now, $\phi^{-1}\big(S(\eta,\ovd)\big)\cap B_3$ is empty, so that 
$$V_\beta(\phi^{-1}(S(\eta,\ovd))\leq \left(\sum_{j,k}C(\xi_{j,k})\right)\delta_1^{2+\beta}\delta_2^{2+\beta}\delta_3^{2+\beta}$$
(observe that the sum is finite). We conclude by applying Lemma \ref{lem:continuityh2}.

\medskip

We now prove that the conditions arising in Theorem \ref{thm:main} are necessary to ensure 
continuity of $C_\phi$ on $H^2(\DD^3)$. For points $\xi\in\TT^3$ such that $\phi(\xi)\in\TT^3,$ this is a consequence of \cite{KSZ08}.
For points $\xi\in\TT^3$ such that there exists $I\subset\{1,2,3\},$ $\phi_I(\xi)\in\TT^2,$
this follows from Theorem \ref{thm:necconditions}.

\section{Examples}

We shall now exhibit examples showing the different cases which can occur with $\phi_I(\xi)\in\TT^2$, $|I|=2$, $\nabla \phi_{i_1}$ and $\nabla\phi_{i_2}$ are linearly
dependent. In what follows, $\veps_0$ (resp. $\veps_1$) is the positive real number coming from Lemma \ref{lem:eps0} (resp. Lemma \ref{lem:eps1}).

\begin{example}
 Let $\veps\in[-\veps_0,\veps_0]$ and 
 \begin{align*}
 \phi(z)&=\Bigg(\frac12\left(z_1z_2+\left(\frac{1+z_3}2\right)+i\veps(z_3-1)^2\right), 
 \frac12\left(z_1z_2+\left(\frac{1+z_3}2\right)\right),0\Bigg).
 \end{align*}
 Then $\phi$ maps $\DD^3$ into itself and $C_\phi$ is continuous on $H^2(\DD^3)$ if and only if $\veps\neq 0.$
\end{example}
\begin{proof}
 Lemma \ref{lem:eps0} immediately implies that $\phi$ is a self-map of $\DD^3$ and that, for $I=\{1,2\}$, $\phi_I(\xi)\in\TT^2\implies \xi=(e^{i\alpha},e^{-i\alpha},1)$
 for some $\alpha\in\mathbb R$. By circular symmetry, we may assume that $\xi=e$. 
 Now,
 \begin{align*}
  \Imm \phi_1(\theta)&=\frac14\left(\theta_1+\theta_2+\theta_3\right)-\veps\frac{\theta_3^2}2+O\left(|\theta_1|^3+|\theta_2|^3+|\theta_3|^3\right)\\
  \Imm \phi_2(\theta)&=\frac14\left(\theta_1+\theta_2+\theta_3\right)+O\left(|\theta_1|^3+|\theta_2|^3+|\theta_3|^3\right)\\
  \Ree \phi_j(\theta)&=1-\frac 14(\theta_1+\theta_2)^2-\frac18\theta_3^2+O\left(|\theta_1|^3+|\theta_2|^3+|\theta_3|^3\right).
 \end{align*}
Therefore, $s(\phi,I,e)=2$ and $r(\phi,I,e)=(1,0)$ provided $\veps<0$, $(0,1)$ provided $\veps>0$, $(0,0)$ provided $\veps=0$.
Hence the conclusion follows from Theorem \ref{thm:main}.
\end{proof}

It could be observed that, for $\veps=0,$ $C_\phi$ is continuous on $A^2(\DD^3)$.

\begin{remark}
 In \cite{Ko22}, it is written that first order conditions, that is conditions on the first order derivative of $\phi$,
 are not sufficient to characterize the continuity of $C_\phi$ on $H^2(\DD^3)$. The argument is based on the functions
 $\phi(z)=(z_1z_2z_3,z_1z_2z_3,0)$ and $\psi(z)=((z_1+z_2+z_3)/3,(z_1+z_2+z_3)/3,0)$ and essentially reduces to the following fact:
 $\phi_{\{1,2\}}(e)=\psi_{\{1,2\}}(e)=(1,1)$, $\nabla \phi(e)=\nabla \psi(e)$, $C_\psi$ is continuous on $H^2(\DD^3)$
 whereas $C_\phi$ is not. However, there are other points $\xi\in\TT^2$ such that $\phi(\xi)\in\TT^2\times\DD,$
 which is not the case for $\psi$. We think that our example is more convincing to illustrate that the first order conditions are not sufficient
 to characterize the continuity of a composition operator on $H^2(\DD^3)$. Indeed, all maps of our family have the same
 points $\xi\in\overline{\DD}^3$ such that $\phi_{\{1,2\}}(\xi)\in\TT^2\times\DD,$ and the same first order derivative
 at these points.
\end{remark}

\begin{example}
 Let $a,b,c\in [-\veps_1,\veps_1]$ and consider $\phi$ defined on $\DD^3$ by
 $$\phi(z)=(F_0(z_1)F_0(z_2)F_0(z_3),F_a(z_1)F_b(z_2)F_c(z_3),0)$$
 where
 $$F_\veps(z)=\frac{3+6z-z^2}8+2i\veps(z-1)^2-i\veps(z-1)^3.$$
 Then $\phi$ is a self-map of $\DD^3$ and $C_\phi$ is continuous on $H^2(\DD^3)$ if and only if $ab+bc+ca>0.$
\end{example}
\begin{proof}
 Lemma \ref{lem:eps1} implies that $\phi$ is a self-map of $\DD^3$ and that, for $I=\{1,2\}$,
 $\phi_I(\xi)\in\TT^2\iff \xi=e.$ We observe that
 \begin{align*}
  \Imm \phi_1(\theta)&=\frac12(\theta_1+\theta_2+\theta_3)+O\left(|\theta_1|^3+|\theta_2|^3+|\theta_3|^3\right)\\
  \Imm \phi_2(\theta)&=\frac 12(\theta_1+\theta_2+\theta_3)-a\theta_1^2-b\theta_2^2-c\theta_3^2+O\left(|\theta_1|^3+|\theta_2|^3+|\theta_3|^3\right)\\
  \Ree \phi_j(\theta)&=1-\frac 18(\theta_1+\theta_2+\theta_3)^2+O\left(|\theta_1|^3+|\theta_2|^3+|\theta_3|^3\right).
 \end{align*}
Hence we now have $s(\phi,I,e)=1$. To compute $r(\phi,I,e)$ we do the change of variables $(\theta_1,\theta_2,\theta_3)\mapsto (\theta_1-\theta_2-\theta_3,\theta_2,\theta_3)$ so that,
in these new coordinates,
\begin{align*}
 \Imm \phi_2(\theta)&=\frac 12\theta_1-a\theta_1^2+2a\theta_1\theta_2+2a\theta_1\theta_3-\left((a+b)\theta_2^2+(a+c)\theta_3^2+2a\theta_2\theta_3\right)\\
 &\quad\quad+O\left(|\theta_1|^3+|\theta_2|^3+|\theta_3|^3\right).
\end{align*}
Hence, using the notations of Section \ref{sec:statement}, 
$$R(\theta)=(a+b)\theta_2^2+2a\theta_2\theta_3+(a+c)\theta_3^2.$$
The signature of $R$ is $(2,0)$ or $(0,2)$ if and only if 
$$4a^2-4(a+b)(a+c)<0\iff ab+bc+ca>0.$$
\end{proof}

It remains to provide an example with $s(\phi,I,\xi)=3$. This is easier. Indeed, consider 
$$\phi(z)=\left(\frac{z_1+z_2+z_3}3,\frac{z_1+z_2+z_3}3,0\right).$$
Then $\phi$ is a self-map of $H^2(\DD^3)$. Let $I=\{1,2\}$ and $\xi\in\TT^3,$
so that $\phi_I(\xi)\in\TT^2.$ By symmetry, we may assume that $\xi=e.$
Then for $j=1,2,$
\begin{align*}
   \Imm \phi_j(\theta)&=\frac13(\theta_1+\theta_2+\theta_3)+O\left(|\theta_1|^2+|\theta_2|^2+|\theta_3|^2\right)\\
   \Ree \phi_j(\theta)&=1-\frac 16\left(\theta_1^2+\theta_2^2+\theta_3^2\right)+O\left(|\theta_1|^3+|\theta_2|^3+|\theta_3|^3\right).
\end{align*}
Therefore $s(\phi,I,e)=3$ and $C_\phi$ defines a bounded operator on $H^2(\DD^3)$. Observe that since $\phi$ is an affine map,
this was already known from the algorithm of \cite{BAYPOLY}.

\section{Compactness issues}

\subsection{A general compactness theorem}
Our aim now is to provide necessary and sufficient conditions for a composition operator
$C_\phi$ acting on $H^2(\DD^d)$ with $d=2$ or $3$ to be compact.
We first state a general result implying the compactness of $C_\phi$ on $H^2(\DD^d)$ by only estimating
the $V_\beta$-volume of preimages of Carleson boxes, in the spirit of Lemma \ref{lem:continuityh2}.
It will rely on the following precise estimate of the Carleson embedding, which is Proposition 9.3 in \cite{BAYPOLY}.
\begin{lemma}\label{lem:precisecarlesonembedding}
Let $d\geq 1.$ There exists $C(d)>0$ such that, 
for any $\beta\in(-1,0]$, for any $\mu$ a finite nonnegative Borel measure on $\overline{\DD^d}$ satisfying:
 \begin{eqnarray}\label{EQCARLESONPRECIS}
  \exists C_\mu>0,\ \forall\xi\in\TT^d,\ \forall\ovd\in(0,2]^d,\ \mu\big(S(\xi,\ovd)\big)\leq C_\mu V_\beta\big(S(\xi,\ovd)\big),
 \end{eqnarray}
 then for any $f\in A_\beta^2(\DD^d)$, 
 $$\left(\int_{\overline{\DD^d}}|f(z_1,\dots,z_d)|^2d\mu\right)^{1/2}\leq C(d)C_\mu\|f\|_{A^2_\beta(\DD^d)}.$$
\end{lemma}

\begin{theorem}\label{thm:compactnessh2}
 Let $\phi:\DD^d\to\DD^d$ be holomorphic. Let us assume that there exist $a>0$ and $\veps:(0,a]\to(0,+\infty)$
 with $\lim_{0^+}\veps=0$ and such that, for all $\ovd\in(0,a]^d,$ for all $\beta\in(-1,0],$ 
 for all $\eta\in\TT^d,$ 
 \begin{equation}\label{eq:compact}
   V_\beta(\phi^{-1}(S(\eta,\ovd)))\leq \veps(\inf(\delta_i:i=1,\dots,d))\delta_1^{2+\beta}\cdots \delta_d^{2+\beta}.
 \end{equation}
 Then $C_\phi$ is compact on $H^2(\DD^d)$.
\end{theorem}
\begin{proof}
 Our strategy, to prove compactness is to approach
 $C_\phi$ by an operator $T_N$ with rank $\leq N^d$, for all $N\in\NN,$  such that, for all $\beta\in(-1,0]$, $\|C_\phi-T_N\|_{\mathcal L(A_\beta^2)}\leq\omega_N$
 with $\omega_N$ which does not depend on $\beta$ and $\lim_{N\to +\infty}\omega_N=0.$
 Letting $\beta$ to $-1$, we then find a sequence of operators of finite rank which converges to $C_\phi$
 in $\mathcal L(H^2(\DD^d))$ leading to the compactness of $C_\phi$. Therefore, the proof is related to the problem
 of finding an upper bound for the approximation numbers of $C_\phi$ and we are inspired by the proof of \cite[Theorem 5.1]{LQR12}.
 
 Let us proceed with the proof. For the sake of simplicity, we shall do it only for $d=2$. As explained in Section \ref{sec:howto},
 we may assume that \eqref{eq:compact} is satisfied for $W(\eta,\ovd)$ instead of $S(\eta,\ovd)$. A standard argument shows that 
 we may assume that $\veps$ is nondecreasing. We may also observe that 
 the assumptions imply that $C_\phi$ is continuous on $A^2_\beta(\DD^2)$, $\beta\in(-1,0]$ as well as on $H^2(\DD^2)$,
 with $\sup_{\beta\in(-1,0]}\|C_\phi\|_{\mathcal L(A^2_\beta)}<+\infty$ (see Section 9.2 of \cite{BAYPOLY}).
 
 Let $N\in\NN$ and for $f=\sum_{\alpha\in\NN^2}a_\alpha z_1^{\alpha_1}z_2^{\alpha_2},$ let us consider
 $P_N(f)=\sum_{|\alpha_i|<N}a_\alpha z_1^{\alpha_1}z_2^{\alpha_2},$ whose rank is $N^d.$ We then set $T_N=C_\phi\circ P_N$.
Assume for a while that $f$ is a polynomial and let $h\in(0,1)$. We set $V_{\phi,\beta}$ as the image of $V_\beta$ under $\phi$
and $\mu_{h,\phi,\beta}$ its restriction to $\DD^2\backslash (1-h)\DD^2.$ We may write
\begin{eqnarray}
 \|C_\phi(f)-T_N(f)\|^2_{\beta}&=&\int_{\DD^2}|(f-P_N(f))\circ\phi|^2 dV_\beta \nonumber \\
 &=&\int_{\DD^2} |f-P_N(f)|^2 dV_{\phi,\beta} \nonumber \\
 &=&\int_{(1-h)\DD^2} |f-P_N(f)|^2 dV_{\phi,\beta} + \int_{\DD^2} |f-P_N(f)|^2 d\mu_{h,\phi,\beta} \label{eq:pfcompact1}
\end{eqnarray}
Let us first handle the second term. We shall apply Lemma \ref{lem:precisecarlesonembedding} with $\mu=\mu_{h,\phi,\beta}.$
Let $\eta\in\TT^2$ and $\ovd\in(0,2]^2.$ If $\delta_1\leq h$ or $\delta_2\leq h,$ then
$$\mu_{h,\phi,\beta}(W(\eta,\ovd))\leq V_{\phi,\beta}(W(\eta,\ovd))\leq \veps(h) \delta_1^{2+\beta}\delta_2^{2+\beta}.$$
If $\delta_1>h$ and $\delta_2>h,$ then we cover $W(\eta,\ovd)\cap (\DD^2\backslash(1-h)\DD^2)$ as follows. Assume for instance
that $\eta=(1,1)$ and let $p=\lfloor \frac{2\delta_1}{h}\rfloor +1$, $q=\lfloor \frac{2\delta_2}h\rfloor+1$. 
Define $(\theta_j)_{j=1,\dots,p+q}$ by
$$\begin{array}{llll}
   \theta_1=-\delta_1,&\delta_2=-\delta_1+h,&\dots,&\delta_p=-\delta_1+(p-1)h\\
   \theta_{p+1}=-\delta_2,&\delta_{p+2}=-\delta_2+h,&\dots,&\delta_{p+q}=-\delta_2+(q-1)h.
  \end{array}$$
We also set, for $j=1,\dots,p,$ $\eta(j)=(e^{i\theta_j},1)$ and $\ovd(j)=(h,\delta_2)$,
whereas, for $j=p+1,\dots,p+q,$ we set $\eta_j=(1,e^{i\theta_j})$ and $\ovd(j)=(\delta_1,h).$
We claim that 
$$W(\eta,\ovd)\cap (\DD^2\backslash(1-h)\DD^2)\subset\bigcup_{j=1}^{p+q}W(\eta(j),\ovd(j)).$$
Indeed, pick $z=(r_1e^{it_1},r_2e^{it_2})$ in $W(\eta,\ovd)\cap (\DD^2\backslash(1-h)\DD^2).$
We may assume $r_1\geq 1-h.$ Then there exists $j\in\{1,\dots,p\}$ such that $t_1\in(\theta_j,\theta_j+h)$
and it is plain that $z\in W(\eta(j),\ovd(j))$. 
 
Hence, provided $\delta_1>h$ and $\delta_2>h$
\begin{align*}
 \mu_{h,\phi,\beta}(W(\eta,\ovd))&=\mu_{h,\phi,\beta}(W(\eta,\ovd)\cap (\DD^2\backslash(1-h)\DD^2))\\
 &\leq \sum_{j=1}^{p+q}\mu_{h,\phi,\beta}(W(\eta(j),\ovd(j))\\
 &\leq p\veps(h)h^{2+\beta}\delta_2^{2+\beta}+q\veps(h)\delta_1^{2+\beta}h^{2+\beta}\\
 &\leq \veps(h)\left(\frac{4\delta_1}h h^{2+\beta}\delta_2^{2+\beta}+\frac{4\delta_2}h \delta_1^{2+\beta}h^{2+\beta}\right)\\
 &\leq 8\veps(h)\delta_1^{2+\beta}\delta_2^{2+\beta}.
\end{align*}
By applying Lemma \ref{lem:precisecarlesonembedding}, there exists $C>0$ which does not depend neither on $\beta\in(-1,0]$
nor on $h\in(0,1)$ nor on $N\in\mathbb N$ such that, for all holomorphic polynomials $f,$
$$\int_{\DD^2}|f-P_N(f)|^2 d\mu_{h,\phi,\beta}\leq C \veps^2(h) \|f-P_N(f)\|^2_\beta\leq C\veps^2(h)\|f\|^2_\beta.$$
We turn to the first term of \eqref{eq:pfcompact1}. We decompose $f-P_N(f)$ as 
\begin{align*}
 (f-P_N(f))(z)&=\sum_{k,l\geq N}a_{k,l}z_1^k z_2^l+\sum_{k<N}\sum_{l\geq N}a_{k,l}z_1^k z_2^l+\sum_{k\geq N}\sum_{l<N}a_{k,l}z_1^kz_2^l\\
 &=:z_1^Nz_2^N s_1(z)+z_2^N s_2(z)+z_1^N s_3(z).
\end{align*}
Therefore, 
\begin{align*}
 \int_{(1-h)\DD^2} |f-P_N(f)|^2 dV_{\phi,\beta}&\leq (1-h)^{2N}\int_{\DD^2}|s_1|^2 dV_{\phi,\beta}+(1-h)^{N}\int_{\DD^2}|s_2|^2 dV_{\phi,\beta}\\
 &\quad\quad\quad+(1-h)^{N}\int_{\DD^2}|s_3|^2 dV_{\phi,\beta}\\
 &\leq C'\left((1-h)^{2N}\|s_1\|_\beta^2+(1-h)^N \|s_2\|_\beta^2+(1-h)^N \|s_3\|_\beta^2\right)
\end{align*}
where $C'=\sup\left(\|C_\phi\|^2_{\mathcal L(A_\beta^2)}:\ \beta\in(-1,0]\right)<+\infty.$
Now,
\begin{align*}
 \|s_1\|_\beta^2&=\sum_{k,l\geq N}\frac{|a_{k,l}|^2}{(k-N+1)^{\beta+1}(l-N+1)^{\beta+1}}\\
 &=\sum_{k,l\geq N}\frac{|a_{k,l}|^2}{(k+1)^{\beta+1}(l+1)^{\beta+1}}\times \frac{(k+1)^{\beta+1}(l+1)^{\beta+1}}{(k-N+1)^{\beta+1}(l-N+1)^{\beta+1}}\\
 &\leq (N+1)^{2+2\beta}\|f\|_{\beta}^2.
\end{align*}
Similarly,
$$\|s_2\|_\beta^2,\|s_3\|_\beta^2\leq (N+1)^{1+\beta}\|f\|_\beta^2.$$
Summarizing what we have done, since $\beta\leq 0,$ we have obtained 
$$\|C_\phi(f)-T_N(f)\|_\beta^2\leq C''\left(\veps^2(h)+(N+1)^2 (1-h)^{2N}+(N+1)(1-h)^N\right)\|f\|_\beta^2$$
where $C''>0$ only depends on $C$ and $C'$.
We choose $h_N=1/\log(N+1)$ and we set
$$\omega_N= C''\left(\veps^2(h_N)+(N+1)^2 (1-h_N)^{2N}+(N+1)(1-h_N)^N\right)$$
which goes to $0$ as $N$ tends to $+\infty$. We get, for all polynomials $f$ and all $\beta\in(-1,0]$,
$$\|C_\phi(f)-T_N(f)\|^2_\beta\leq \omega_N \|f\|_\beta^2.$$
Letting $\beta$ to $-1$ yields $\|C_\phi-T_N\|_{\mathcal L(H^2)}\leq \sqrt{\omega_N}$, so that $C_\phi$ is compact
as a limit of finite rank operators.
\end{proof}
We shall also need a necessary condition for compactness. The following corollary of \cite[Theorem 2.3]{Jaf91} will be enough for us.
\begin{lemma}\label{lem:noncompacth2}
Let $\phi:\DD^d\to\DD^d$ be holomorphic and belonging to $\mathcal C(\overline{\DD^d})$. Assume that $C_\phi$ is compact on $H^2(\DD^d)$. For any sequence $(\ovd(n))$
of $(0,2]^d$ such that $\delta_1(n)\cdots\delta_d(n)$ goes to $0,$
for any $\eta\in\TT^d,$
\begin{equation}\label{eq:noncompactness} 
\frac{\sigma_d(\{e^{i\theta}\in\TT^d:\ \phi(\theta)\in S(\eta,\ovd(n))\})}{\delta_1(n)\cdots \delta_d(n)}\xrightarrow{n\to+\infty}0.
\end{equation}
\end{lemma}

A first application of the work of this section can be done for affine symbols. 
If $\phi:\DD^d\to\DD^d$ is affine, for any $\xi\in\TT^d$ such that $\phi(\xi)\notin \DD^d,$ two families of trees $\mathbf R^{-1}_{\xi,I,J}$ and $\mathbf L^{-1}_{\xi,I,J}$ were introduced in \cite{BAYPOLY} and it was shown that $C_\phi$ is continuous on $H^2(\DD^d)$ if and only if $\mathbf R^{-1}_{\xi,I,J}\leq \mathbf L_{\xi,I,J}^{-1}$
for all possible choices of $\xi,I,J$. Using the methods of \cite{BAYPOLY} combined with Theorem \ref{thm:compactnessh2} and Lemma \ref{lem:noncompacth2}, we can prove that 
$C_\phi$ is compact on $H^2(\DD^d)$ if and only if $\mathbf R^{-1}_{\xi,I,J}< \mathbf L_{\xi,I,J}^{-1}$
for all possible choices of $\xi,I,J$. Details are left to the interested reader.

\subsection{Compact composition operators on $H^2(\DD^2)$}
We shall use the results of the previous subsection to almost characterize compact composition operators on $H^2(\DD^2)$. Recall (see \cite[Theorem 6]{Ko22}) that
for $\phi\in \mathcal O(\DD^2,\DD^2)\cap\mathcal C^1(\overline{\DD^2}),$
$C_\phi$ is bounded on $H^2(\DD^2)$ if and only if $\mathrm d\phi(\xi)$ 
is invertible for all $\xi\in\TT^2$ such that $\phi(\xi)\in\TT^2.$

We shall need to introduce two new definitions. We start from $\varphi\in\mathcal O(\DD^2,\DD)\cap\mathcal C^3( \overline{\DD^2})$ and let $\xi\in\TT^2$
such that $\varphi(\xi)\in\TT.$ Without loss of generality, we will assume $\xi=e$ and $\varphi(\xi)=1$. We will say that $\varphi$ is truely two-dimensional at $\xi$ provided
\begin{itemize}
\item $\frac{\partial\varphi}{\partial z_1}(e)\neq 0,$ $\frac{\partial\varphi}{\partial z_2}(e)\neq 0;$
\item let us write $\Imm \varphi(\theta)=L(\theta)+G(\theta)$ where $L$ is a linear
form and $G(\theta)=O(\theta_1^2+\theta_2^2)$. The map $U(\theta_1,\theta_2)=(L(\theta)+G(\theta),\theta_2)$ defines a diffeomorphism between two neighbourhoods $\mathcal U$ and $\mathcal V$ 
of $0$ in $\mathbb R^2$ since $\frac{\partial L}{\partial \theta_1}(0)\neq 0.$ We then may write
$$\Ree \varphi\circ U^{-1}(u_1,u_2)=1+ H(u_1,u_2).$$
We shall require, for a truely two-dimensional function at $\xi,$ that there does not exist $h$ such that $H(u_1,u_2)=h(u_1)$ in a neighbourhood of $0.$ 
\end{itemize}

We will say that $\varphi$ is truely two-dimensional at $\xi$ of order $2$ provided, writing
$\Ree \varphi(\theta)=1-Q(\theta)+o(\theta_1^2+\theta_2^2)$ with $Q$ a quadratic form,
that the signature of $Q$ is $(2,0)$. Observe that a truely two-dimensional function at $\xi$ 
of order $2$ is automatically truely two-dimensional at $\xi.$ Indeed, $\frac{\partial\varphi}{\partial z_1}(e)\neq 0$ and $\frac{\partial\varphi}{\partial z_2}(e)\neq 0$, otherwise
$\varphi$ would not depend on one of the two variables and the signature of $Q$ could not be $(2,0)$. Moreover, if we do the change of variables described above, we may write
$$\Ree \varphi\circ U^{-1}(u_1,u_2)=1-\tilde Q(u_1,u_2)+H(u_1,u_2)$$
where $\textrm{signature}(\tilde Q)=(2,0)$ and $H(u_1,u_2)=o(u_1^2+u_2^2)$. 

\begin{example}\label{ex:compact1}
Let $\varphi(z)=z_1z_2$. Then $\varphi$ is not truely two-dimensional at $e$. 
Indeed, we may write $\Imm \varphi(\theta)=f(\theta_1+\theta_2)$ and
$\Ree \varphi(\theta)=1+g(\theta_1+\theta_2)$. Let $U(\theta_1,\theta_2)=(f(\theta_1+\theta_2),\theta_2)$. A straightforward computation shows that $U^{-1}(u_1,u_2)=(f^{-1}(u_1)-u_2,u_2)$ so that 
$$\Ree\varphi\circ U^{-1}(u_1,u_2)=1+g(f^{-1}(u_1)).$$
Heuristically speaking, $\varphi$ is not truely two-dimensional since $\varphi(\theta)$
only depends on the single variable $\theta_1+\theta_2$.
\end{example}

\begin{example}\label{ex:compact2}
Let $\varphi(z)=\frac{z_1+z_2}2$. Then $\varphi$ is truely two-dimensional at $e$ of order $2$ (in fact, at any $\xi\in\TT^2$ such that $\varphi(\xi)\in\TT$). Indeed,
$$\Ree \varphi(\theta)=1-\frac 14(\theta_1^2+\theta_2^2)+o(\theta_1^2+\theta_2^2).$$
\end{example}

\medskip

The interest of working with a truely two-dimensional function of order 2 comes from the following result, which should be compared with Theorem \ref{thm:volumelinearlydependent}.
\begin{theorem}\label{thm:compactnessvolume}
Let $\varphi\in\mathcal O(\DD^2,\DD)\cap\mathcal C^3(\overline{\DD^3}).$ 
Assume that there exists $\xi\in\TT^2$ such that $\varphi(\xi)\in\TT$ and that $\varphi$
is truely two-dimensional at $\xi$ of order $2.$ Then there exist a neighbourhood $\mathcal U$ of $\xi$ and $C>0$ such that, for all $\beta\in(-1,0]$, for all $\eta\in\TT,$ for all $\delta>0,$
$$V_\beta\left(\left\{z\in\DD^2\cap\mathcal U:\ |\varphi(z)-\eta|<\delta\right\}\right)\leq
C\delta^{2+\beta+\frac 12}.$$
\end{theorem}
\begin{proof}
We may assume $\xi=e$ and $\varphi(\xi)=1$. We mimic the proof of Theorem \ref{thm:volumelinearlydependent}. We start with the following fact.

\medskip

{\bf Fact.} There exist $C>0$ and a neighbourhood $\mathcal O$ of $0$ in $\mathbb R^2$ such that, for all $\delta>0$, for all $\eta\in\TT,$
\begin{equation}
\lambda_2\left(\left\{\theta\in\mathcal O:\ |\varphi(\theta)-\eta|\leq \delta\right\}\right)\leq C\delta^{\frac 32}.
\end{equation}

\bigskip

We write
\begin{align*}
\Imm \varphi(\theta)&=L(\theta)+G(\theta)\\
\Ree \varphi(\theta)&=1-Q(\theta)+o(\theta_1^2+\theta_2^2)
\end{align*}
with $G(\theta)=o(|\theta_1|+|\theta_2|)$ and the signature of $Q$ equal to $(2,0)$.
By a linear change of variables, which does not modify the signature of $Q$, we may  assume that $L(\theta)=\theta_1.$ We then set $U(\theta)=(L(\theta)+G(\theta),\theta_2).$ 
There exist two neighbourhoods $\mathcal O_1$ and $\mathcal O_2$ of $0$ in $\mathbb R^2$ and $C_1>0$ such that $U$ is a diffeomorphism from $\mathcal O_1$ onto $\mathcal O_2$ and 
$$\left\{
\begin{array}{rcll}
|\det(\mathrm dU)|&\leq& C_1&\textrm{on }\mathcal O_1\\
|\det(\mathrm dU^{-1})|&\leq& C_1&\textrm{on }\mathcal O_2.
\end{array}\right.$$
Since $\mathrm dU(0)=\mathrm{Id},$ we can still write
$$\Ree \varphi\circ U^{-1}(u_1,u_2)=1-Q(u)+o(u_1^2+u_2^2).$$
We then apply the parametrized Morse lemma: there exist two neighbourhoods $\mathcal O'_2\subset\mathcal O_2$ and $\mathcal O_3$ of $0$, a constant $C_2>0$ and a diffeomorphism $V:\mathcal O_2'\to\mathcal O_3$, $V(u_1,u_2)=(u_1,V_2(u_1,u_2))$ and
a $\mathcal C^1$-map $h:\mathbb R\to\mathbb R$ such that, for all $v\in\mathcal O_3,$
$$\Ree \varphi\circ U^{-1}\circ V^{-1}(v_1,v_2)=1-v_2^2+h(v_1)$$
whereas we have still
$$\Imm \varphi\circ U^{-1}\circ V^{-1}(v_1,v_2)=v_1$$
as well as
$$\left\{
\begin{array}{rcll}
|\det(\mathrm dV)|&\leq& C_2&\textrm{on }\mathcal O'_2\\
|\det(\mathrm dV^{-1})|&\leq& C_2&\textrm{on }\mathcal O_3.
\end{array}\right.$$
Therefore, for any $\theta\in U^{-1}(\mathcal O'_2)=:\mathcal O$, for all $\eta\in\TT,$ for all $\delta>0,$
setting $v=V\circ U(\theta),$ we have 
$$\left\{
\begin{array}{rcl}
|v_1-\Imm (\eta)|&<&\delta\\
|1-v_2^2+h(v_1)-\Ree (\eta)|&<&\delta.
\end{array}\right.
$$
Fubini's theorem yields
$$\lambda_2\left(V\circ U\left(\mathcal O\cap \varphi^{-1}(S(\eta,\delta))\right)\right)\lesssim \delta^{3/2}$$
which in turn implies the fact. We finally deduce Theorem \ref{thm:compactnessvolume} from the fact following the method used in Theorem \ref{thm:volumelinearlydependent}.
\end{proof}

We are now ready to almost characterize compact composition operators on $H^2(\DD^2).$

\begin{theorem}\label{thm:compactnessdim2}
Let $\phi\in\mathcal O(\DD^2,\DD^2)\cap\mathcal C^3(\overline{\DD^2})$. If $C_\phi$ 
defines a compact composition operator on $H^2(\DD^2),$ then 
\begin{enumerate}
\item $\phi(\TT^2)\cap \TT^2=\varnothing;$
\item for all $\xi\in\TT^2,$ for all $j\in\{1,2\}$ such that $\phi_j(\xi)\in\TT,$ 
$\phi_j$ is truely two-dimensional at $\xi.$
\end{enumerate}
Conversely, assume that 
\begin{enumerate}
\item $\phi(\TT^2)\cap \TT^2=\varnothing;$
\item for all $\xi\in\TT^2,$ for all $j\in\{1,2\}$ such that $\phi_j(\xi)\in\TT,$ 
$\phi_j$ is truely two-dimensional at $\xi$ of order $2$.
\end{enumerate}
Then $C_\phi$ is compact on $H^2(\DD^2)$.
\end{theorem}
\begin{proof}
We first prove the necessary condition and we start from a compact composition operator
$C_\phi$ on $H^2(\DD^2)$. First assume that $\phi(\TT^2)\cap \TT^2\neq\varnothing,$ 
for instance $\phi(e)=e.$ Let $\delta>0$ be small and set $\ovd=(\delta,\delta).$ 
By the mean value theorem, there exists $c>0$ (which does not depend on $\delta$ provided $\delta$ is small enough) such that 
$$\left\{e^{i\theta}:\ \phi(\theta)\in S(e,\ovd)\right\}\supset [-c\delta,c\delta]^2.$$
This prevents \eqref{eq:noncompactness}  to be true and $C_\phi$ cannot be compact. 
Assume now that there exists $\xi\in\TT^2$ and $j\in\{1,2\}$ such that $\phi_j(\xi)\in\TT$ 
and $\phi_j$ is not truely two-dimensional at $\xi,$ say $j=1$, $\xi=e$ and $\phi_1(e)=1.$
If $\frac{\partial \phi_1}{\partial z_2}(e)=0,$ then $\phi_1$ only depends on $z_1$ and we
set $\ovd=(\delta,2)$ for small values of $\delta.$ Again, there exists $c>0$ such that
\begin{align*}
\left\{e^{i\theta}:\ \phi(\theta)\in S(e,\ovd)\right\}&=\left\{e^{i\theta}:\ |\phi_1(\theta)-1|<\delta\right\}\\
&\supset [-c\delta,c\delta]\times [-1,1]
\end{align*}
and again \eqref{eq:noncompactness} cannot be true. If $\frac{\partial \phi_1}{\partial z_1}(e)\neq 0$ and $\frac{\partial \phi}{\partial z_2}(e)\neq 0,$ since $\phi_1$
is not truely two-dimensional at $e,$ let $\mathcal U$ and $\mathcal V$ be two neighbourhoods of $0,$ $C>0$ and $U$ a $\mathcal C^1$-diffeomorphism from $\mathcal U$ onto $\mathcal V$
such that 
$$\left\{
\begin{array}{rcll}
|\det(\mathrm dU)|&\leq& C_1&\textrm{on }\mathcal O_1\\
|\det(\mathrm dU^{-1})|&\leq& C_1&\textrm{on }\mathcal O_2
\end{array}\right.$$
and
$$\left\{
\begin{array}{rcl}
\Imm \phi_1\circ U^{-1}(u_1,u_2)&=&u_1\\
\Ree \phi_1\circ U^{-1}(u_1,u_2)&=&1+h(u_1)
\end{array}\right.$$
with $h(u_1)=O(u_1)$. Therefore there exists $c>0$ such that, for all $\delta>0$,
for all $(u_1,u_2)\in ([-c\delta,c\delta]\times[-1,1])\cap\mathcal V,$
$$|\phi_1\circ U^{-1}(u_1,u_2)-1|<\delta.$$
This yields 
$$U^{-1}\big(([-c\delta,c\delta]\times[-1,1])\cap\mathcal V\big)\subset\phi^{-1}(S(e,\ovd))$$
for $\ovd=(\delta,2)$ which prevents \eqref{eq:noncompactness} to be true.

Let us turn to the positive part. It will follow by applying Theorem \ref{thm:compactnessvolume} with a suitable covering argument. Since $\phi(\TT^2)\cap\TT^2=\varnothing,$ there exists
$\veps_1\in(0,1)$ such that, for all $\xi\in\overline{\DD^2},$ there exists $j\in\{1,2\}$ with $|\phi_j(\xi)|\leq 1-2\veps_1.$
We then set 
$$A_1=\left\{\xi\in\DD^2:\ \exists j\in\{1,2\},\ \phi_j(\xi)\in\TT\right\}.$$
Let $\xi\in A_1$ and assume for instance that $\phi_1(\xi)\in\TT.$
Let $\mathcal U(\xi)$ and $C(\xi)$ be given by Theorem \ref{thm:compactnessvolume}.
Restricting $\mathcal U(\xi)$ if necessary, we may assume that $|\phi_2(z)|<1-\veps_1$ for all $z\in\mathcal U(\xi)$. Let $\beta\in(-1,0]$, $\eta\in\TT^2$ and $\ovd\in(0,2]^2.$
Assume first that $\delta_2\geq\veps_1.$ Then 
\begin{align}
V_\beta\left(\phi^{-1}(S(\eta,\ovd))\cap\mathcal U\right)&\leq V_\beta\left(\left\{z\in\overline{\DD^2}\cap \mathcal U:\ |\phi_1(z)-\eta_1|<\delta_1\right\}\right)\nonumber \\
&\leq C(\xi)\delta_1^{2+\beta+\frac 12}\nonumber \\
&\leq \frac{C(\xi)}{\veps_1^{2+\beta+\frac 12}}\delta_1^{2+\beta+\frac 12}\delta_2^{2+\beta+\frac 12}\nonumber \\
&\leq D(\xi) \delta_1^{2+\beta+\frac 12}\delta_2^{2+\beta+\frac 12}\label{eq:fincompact}
\end{align}
where we have set $D(\xi)=C(\xi)/\veps_1^{5/2}.$ On the other hand, if $\delta_2<\veps_1,$ then $\phi^{-1}(S(\eta,\ovd))\cap\mathcal U(\xi)=\varnothing$ so that 
\eqref{eq:fincompact} is verified for all $\ovd\in(0,2]^2.$
By compactness, we cover $A_1$ by finitely many $\mathcal U(\xi_1),\dots,\mathcal U(\xi_p)$
and we set $A_2=\overline{\DD^2}\backslash \bigcup_{k=1}^p \mathcal U(\xi_k).$ There
exists $\veps_2>0$ such that, provided $z\in B_2,$ $|\phi_1(z)|,\ |\phi_2(z)|<1-\veps_2$. 
Therefore, for all $\eta\in\TT^2$ and all $\ovd\in(0,2]^2$ with $\inf(\delta_1,\delta_2)<\veps_2,$ 
$$V_\beta\left(\phi^{-1}(S(\eta,\ovd))\right)\leq \left(\sum_{k=1}^p D(\xi_k)\right)\delta_1^{2+\beta+\frac 12}\delta_2^{2+\beta+\frac 12}.$$
We infer from Theorem \ref{thm:compactnessh2} that $C_\phi$ is compact.
\end{proof}

We can then use Examples \ref{ex:compact1} and \ref{ex:compact2} to provide examples of symbols $\phi\in\mathcal O(\DD^2,\DD^2)$
touching the boundary 
such that $C_\phi$ is continuous on $H^2(\DD^2)$ and $C_\phi$ is compact (resp. not compact),
for instance $\phi(z)=((z_1+z_2)/2,0)$ (resp. $\phi(z)=(z_1z_2,0)$).

\subsection{Compact composition operators on $H^2(\DD^3)$}
We now start a similar study of compact composition operators on $H^2(\DD^3)$. It turns out that the symbols
which we study in depth in Section \ref{sec:volumeestimate} give rise to noncompact composition operators.
As above, we need to introduce several definitions. 
First of all, let $\varphi\in\mathcal O(\DD^3,\DD)\cap\mathcal C^2( \overline{\DD^3})$ and let $\xi\in\TT^3$
such that $\varphi(\xi)\in\TT.$ Without loss of generality, we will assume $\xi=e$ and $\varphi(\xi)=1$. We will say that $\varphi$ is truely at least two-dimensional at $\xi$ provided
\begin{itemize}
\item at least two of $\frac{\partial\varphi}{\partial z_1}(e),$ $\frac{\partial\varphi}{\partial z_2}(e),$ $\frac{\partial\varphi}{\partial z_3}(e)$ are not equal to $0;$
\item let us assume that $\frac{\partial \varphi}{\partial z_1}(e)\neq 0$ and let us write $\Imm \varphi(\theta)=L(\theta)+G(\theta)$ where $L$ is a linear
form and $G(\theta)=O(\theta_1^2+\theta_2^2+\theta_3^2)$. The map $U(\theta)=(L(\theta)+G(\theta),\theta_2,\theta_3)$ defines a diffeomorphism between two neighbourhoods $\mathcal U$ and $\mathcal V$ of
$0$ in $\mathbb R^3$ since $\frac{\partial L}{\partial \theta_1}(0)\neq 0.$ We then may write
$$\Ree \varphi\circ U^{-1}(u)=1+H(u).$$
 We shall require, for a truely at least two-dimensional function at $\xi,$ that $H(u)\neq h(u_1)$ for some function $h.$ 
\end{itemize}

We will say that $\varphi$ is truely at least two-dimensional at $\xi$ of order $2$ provided, writing
$\Ree \varphi(\theta)=1-Q(\theta)+O(\theta_1^2+\theta_2^2+\theta_3^2)$ with $Q$ a quadratic form,
that the signature of $Q$ is $(2,0)$ or $(3,0)$. As before a truely at least two-dimensional function at $\xi$ 
of order $2$ is automatically truely two-dimensional at $\xi.$ 

We will also have to consider symbols $\phi$ such that there exists $I\subset\{1,2,3\}$ with $|I|=2$ and $\phi_I(\TT^3)\cap\TT^2\neq \varnothing.$
Here are the relevant definitions. Let $\varphi\in \mathcal O(\DD^3,\DD^2)\cap\mathcal C^3( \overline{\DD^3})$ and let $\xi\in\TT^3$
such that $\varphi(\xi)\in\TT^2$.  Without loss of generality, we will assume $\xi=e$ and $\varphi(\xi)=(1,1)$. We will say that $\varphi$ is truely three-dimensional at $\xi$ provided
\begin{itemize}
\item $\nabla\varphi_1(e)$ and $\nabla\varphi_2(e)$ are linearly independent
\item let us assume that $\left(\frac{\partial \varphi_j}{\partial z_k}(e)\right)_{1\leq j,k\leq 2}$ is invertible and let us write
$\Imm \varphi_j(\theta)=L_j(\theta)+G_j(\theta)$ where $L_j$ is a linear
form and $G_j(\theta)=O(\theta_1^2+\theta_2^2+\theta_3^2)$. The map $U(\theta)=(L_1(\theta),L_2(\theta),\theta_3)$ defines a diffeomorphism between two neighbourhoods
of $0$ in $\mathbb R^3$ and  for $j=1,2,$
$$\Ree \varphi_j\circ U^{-1}(u)=1+ H_j(u).$$
We shall require, for a truely three-dimensional function, that $H_j(u)\neq h_j(u_1,u_2)$ for some function $h_j,$ for $j=1,2.$
\end{itemize}

We will say that $\varphi$ is truely three-dimensional at $\xi$ of order $2$ provided that $\nabla\varphi_1(e)$ and $\nabla\varphi_2(e)$ are linearly independent and, writing
$\Ree \varphi_j(\theta)=1-Q_j(\theta)+o(\theta_1^2+\theta_2^2+\theta_3^2)$ with $Q_j$ a quadratic form,
that the signature of both $Q_1$ and $Q_2$ is equal to $(3,0)$. Observe that a truely three-dimensional function at $\xi$ 
of order $2$ is automatically truely three-dimensional at $\xi.$ 

\begin{example}
 The map $\varphi(z)=\left(z_1z_2,\frac{z_1}3+\frac{2z_2}3\right)$ is not truely three-dimensional at $e,$ even if $\nabla\varphi_1(e)$
 and $\nabla\varphi_2(e)$ are linearly independent. On the contrary, the map
 $\varphi(z)=\left(\frac13z_1+\frac 13z_2+\frac13z_3,\frac 14z_1+\frac 12z_2+\frac 14z_3\right)$
 is truely three-dimensional of order $2$ at $e.$
\end{example}

We now give an analogue of Theorem \ref{thm:compactnessvolume} in dimension 3.

\begin{theorem}\label{thm:compactnessvolume3}
\begin{enumerate}[(a)]
 \item Let $\varphi\in\mathcal O(\DD^3,\DD)\cap\mathcal C^3(\overline{\DD^3}).$ 
Assume that there exists $\xi\in\TT^3$ such that $\varphi(\xi)\in\TT$ and that $\varphi$
is truely at least two-dimensional at $\xi$ of order $2.$ Then there exist a neighbourhood $\mathcal U$ of $\xi$ and $C>0$ such that, for all $\beta\in(-1,0]$, 
for all $\eta\in\TT,$ for all $\delta>0,$
$$V_\beta\left(\left\{z\in\DD^3\cap\mathcal U:\ |\varphi(z)-\eta|<\delta\right\}\right)\leq
C\delta^{2+\beta+\frac 12}.$$
\item  Let $\varphi\in\mathcal O(\DD^3,\DD^2)\cap\mathcal C^2(\overline{\DD^3}).$ 
Assume that there exists $\xi\in\TT^3$ such that $\varphi(\xi)\in\TT^2$ and that $\varphi$
is truely three-dimensional at $\xi$ of order $2.$  Then there exist a neighbourhood $\mathcal U$ of $\xi$ and $C>0$ such that, for all $\beta\in(-1,0]$, 
for all $\eta\in\TT^2,$ for all $\ovd\in(0,2]^2,$
$$V_\beta\left(\left\{z\in\DD^3\cap\mathcal U:\ |\varphi_j(z)-\eta_j|<\delta_j,\ j=1,2\right\}\right)\leq C\delta_1^{2+\beta}\delta_2^{2+\beta}\big(\inf(\delta_1,\delta_2)\big)^{1/2}.$$ 
\end{enumerate}
\end{theorem}
\begin{proof}
We shall be very squetchy. For (a), we just need to prove the following: there exist $C>0$ 
and a neighbourhood $\mathcal O$ of $0$ in $\mathbb R^2$ such that, for all $\delta>0,$
for all $\eta\in\TT,$ 
$$\lambda_3\left(\left\{\theta\in\mathcal O:\ |\varphi(\theta)-\eta|\leq\delta\right\}\right)\leq C\delta^{3/2}.$$
We may and shall assume that $\Imm \varphi(\theta)=\theta_1$ and write $\Ree \varphi(\theta)=1-Q(\theta)+o(\theta_1^2+\theta_2+\theta_3^2)$
where $\textrm{signature}(Q)=(2,0)$ or $(3,0)$. In both cases we can do a linear change of variables to assume that $Q(\theta)=\theta_2^2+R(\theta),$
with $\textrm{signature}(R)=(1,0)$ or $(2,0)$. We then apply the parametrized Morse lemma to get new variables $v=(v_1,v_2,v_3)$ and a $\mathcal C^1$-map
$h:\mathcal V\subset\RR^2\to\RR$ such that
\begin{align*}
 \Imm \varphi(v)&=v_1\\
 \Ree \varphi(v)&=1-v_2^2+h(v_1,v_3).
\end{align*}
We finally conclude as in Theorem \ref{thm:compactnessvolume}. Observe that if the signature of $Q$ is $(3,0)$ we
could even get a better estimate by using the parametrized Morse lemma in order to write $\Ree \varphi(v)=1-v_2^2-v_3^2+h(v_1)$,
but this is not useful for our purpose.

Regarding (b), we start by writing
\begin{align*}
 \Imm \varphi_1(\theta)&=L_1(\theta)+G_1(\theta)\\
 \Imm \varphi_2(\theta)&=L_2(\theta)+G_2(\theta)
\end{align*}
with $L_1,L_2$ two independent linear forms. We may assume $L_1(\theta)=\theta_1,$ $L_2(\theta)=\theta_2$ and
using the change of variables $U(\theta)=(L_1(\theta)+G_1(\theta),L_2(\theta)+G_2(\theta),\theta_3),$
we will assume 
\begin{align*}
 \Imm \varphi_1(\theta)&=\theta_1\\
 \Imm \varphi_2(\theta)&=\theta_2.
\end{align*}
Let us assume that $\delta_1\leq\delta_2.$ Since the signature of $Q_1$ is $(3,0)$, we may apply the parametrized Morse lemma
and assume, up to a change of variables, that
$$\Ree \varphi_1(\theta)=1-\theta_3^2+h_1(\theta_1,\theta_2)$$
with $h_1$ a $\mathcal C^1$-function defined in a neighbourhood of $0.$ Arguing as in the proof of
Theorem \ref{thm:compactnessvolume}, we get
$$\lambda_3\left(\left\{\theta\in\mathcal U:\ |\phi_j(\theta)-\eta_j|\leq\delta_j,\ j=1,2\right\}\right)\leq C\delta_1\delta_2\big(\inf(\delta_1,\delta_2)\big)^{1/2}.$$
\end{proof}

We deduce a result giving a large number of examples of compact and noncompact composition operators on $H^2(\DD^3)$.
\begin{theorem}\label{thm:compactnessdim3}
 Let $\phi\in\mathcal O(\DD^3,\DD^3)\cap\mathcal C^3(\overline{\DD^3}).$ If $C_\phi$ defines a compact composition operator on $H^2(\DD^3)$, then
 \begin{enumerate}
  \item $\phi(\TT^3)\cap\TT^3=\varnothing$;
  \item for all $\xi\in\TT^3,$ for all $I\subset\{1,2\}$ with $\phi_I(\xi)\in\TT^2,$
  $\phi_I$ is truely three-dimensional at $\xi$;
  \item for all $\xi\in\TT^3,$ for all $j\in\{1,2,3\}$ such that $\phi_j(\xi)\in\TT,$
  $\phi_j$ is truely at least two-dimensional at $\xi$.
 \end{enumerate}
Conversely, assume that
 \begin{enumerate}
  \item $\phi(\TT^3)\cap\TT^3=\varnothing$;
  \item for all $\xi\in\TT^3,$ for all $I\subset\{1,2\}$ with $\phi_I(\xi)\in\TT^2,$
  $\phi_I$ is truely three-dimensional at $\xi$ of order $2$;
  \item for all $\xi\in\TT^3,$ for all $j\in\{1,2,3\}$ such that $\phi_j(\xi)\in\TT,$
  $\phi_j$ is truely at least two-dimensional at $\xi$ of order $2$.
 \end{enumerate}
 Then $C_\phi$ is compact on $H^2(\DD^3)$.
\end{theorem}
\begin{proof}
 We shall only consider the necessary condition, since the sufficient condition can be proved exactly like the 
 sufficient condition of Theorem \ref{thm:compactnessdim2}. That (1) is mandatory
 follows by the mean value theorem like in the proof of Theorem \ref{thm:compactnessdim2}. More generally,
 for any $d\geq 1$ and any symbol $\phi\in\mathcal O(\DD^d,\DD^d)\cap\mathcal C^1(\overline{\DD^d})$
 such that $\phi(\TT^d)\cap\TT^d\neq\varnothing,$ $C_\phi$ cannot be compact on $H^2(\DD^d)$.
 That (3) is necessary also follows from the arguments of Theorem \ref{thm:compactnessdim2} and thus we just deal with (2).
 Assume that for $I=\{1,2\}$ and $\xi=e,$ $\phi_I(e)=(1,1)$ and $\phi_I$ is not truely three-dimensional at $\xi.$
 If $\nabla\phi_1(e)$ and $\nabla\phi_2(e)$ are not linearly independent, we may assume
 \begin{align*}
  \Imm \phi_1(\theta)&=\kappa_1 \theta_1+O(\theta_1^2+\theta_2^2+\theta_3^2)\\
  \Imm \phi_2(\theta)&=\kappa_2 \theta_1+O(\theta_1^2+\theta_2^2+\theta_3^2)\\
   \Ree \phi_j(\theta)&=1+O(\theta_1^2+\theta_2^2+\theta_3^2),\ j=1,2.
 \end{align*}
Then, for some $c>0,$ for all $\delta>0,$ setting $\ovd=(\delta,\delta,1),$
$$\left\{e^{i\theta}:\ \phi(\theta)\in S(e,\ovd)\right\}\supset [-c\delta,c\delta]\times [-c\delta^{1/2},c\delta^{1/2}]\times [-c\delta^{1/2},c\delta^{1/2}]$$
which prevents \eqref{eq:noncompactness} to be true. If $\nabla\phi_1(e)$ and $\nabla\phi_2(e)$ are linearly independent, we can assume
\begin{align*}
 \Imm\phi_1(\theta)&=\theta_1\\
 \Imm\phi_2(\theta)&=\theta_2\\
 \Ree \phi_j(\theta)&=1+O(\theta_1^2+\theta_2^2),\ j=1,2.
\end{align*}
Then, for some $c>0,$ for all $\delta>0,$ setting $\ovd=(\delta,\delta,1),$
$$\left\{e^{i\theta}:\ \phi(\theta)\in S(e,\ovd)\right\}\supset [-c\delta,c\delta]\times [-c\delta,c\delta]\times [-c,c]$$
which prevents \eqref{eq:noncompactness} to be true.
\end{proof}

\begin{example}
 \begin{enumerate}[(a)]
  \item Let $\phi(z)=\left(z_1z_2,\frac{z_1}3+\frac{2z_2}3,0\right)$. Then $C_\phi$ is a continuous and not compact operator on $H^2(\DD^3)$.
  \item Let $\phi(z)=\left(\frac13z_1+\frac 13z_2+\frac13z_3,\frac 14z_1+\frac 12z_2+\frac 14z_3,0\right)$.
  Then $C_\phi$ is a compact operator on $H^2(\DD^3)$.
 \end{enumerate}

\end{example}

\begin{remark}
 To characterize completely regular symbols leading to compact composition operators on $H^2(\DD^2)$ or $H^2(\DD^3),$
 it seems that we would need to inspect high order terms in the expansion of $\Ree \phi_j$ and $\Imm \phi_j.$
 In that case, some tools are missing (the signature, the parametrized Morse lemma).
 However, for a specific map, a detailed study using the arguments developed here should be possible.
 It is likely that the study of continuity on $H^2(\DD^d)$ with $d\geq 4$ also requires
 to consider more terms in the expansion of $\Ree\phi_j$ and $\Imm\phi_j$.
\end{remark}

\providecommand{\bysame}{\leavevmode\hbox to3em{\hrulefill}\thinspace}
\providecommand{\MR}{\relax\ifhmode\unskip\space\fi MR }
\providecommand{\MRhref}[2]{%
  \href{http://www.ams.org/mathscinet-getitem?mr=#1}{#2}
}
\providecommand{\href}[2]{#2}


\begin{thebibliography}{LQRP12}

\bibitem[Bay11]{BAYPOLY}
F.~Bayart, \emph{Composition operators on the polydisk induced by affine maps},
  J. Funct. Anal. \textbf{260} (2011), 1969--2003.

\bibitem[BG92]{BG92}
J.~W. Bruce and P.~J. Giblin, \emph{Curves and {S}ingularities: a geometrical
  introduction to singularity theory}, Cambridge University press, 1992.

\bibitem[Cha79]{Ch79}
A.~Chang, \emph{Carleson measure on the bi-disc}, Ann. of Math. \textbf{109}
  (1979), 613--620.

\bibitem[CM95]{CoMcl95}
C.~C. Cowen and B.~MacCluer, \emph{Composition operators on spaces of analytic
  functions}, Studies in Advanced Mathematics, CRC Press, 1995.

\bibitem[CSW84]{CSW84}
J.~Cima, C.~Stanton, and W.~Wogen, \emph{{On boundedness of composition
  operators on $H^{2}(B_{2})$}}, Proc. Amer. Math. Soc. \textbf{91} (1984),
  217--222.

\bibitem[Jaf91]{Jaf91}
F.~Jafari, \emph{{Carleson measures in Hardy and weighted Bergman spaces of
  polydiscs}}, Proc. Amer. Math. Soc. \textbf{112} (1991), 771--781.

\bibitem[Kos23]{Ko22}
L~Kosiński, \emph{{Composition operators on the polydisc}}, J. Funct. Anal.
  \textbf{284} (2023), 109801.

\bibitem[KSZ08]{KSZ08}
H.~Koo, M.~Stessin, and K.~Zhu, \emph{Composition operators on the polydisc
  induced by smooth symbols}, J. Funct. Anal. \textbf{254} (2008), 2911--2925.

\bibitem[LQRP12]{LQR12}
D.~Li, H.~Queff\'elec, and L.~Rodríguez-Piazza, \emph{{On approximation
  numbers of composition operators}}, J. Approx. Th. \textbf{164} (2012),
  431--459.

\bibitem[Wog88]{Wo88}
W.~Wogen, \emph{{The smooth mappings which preserve the Hardy space
  $H^2(B_n)$}}, Contributions to operator theory and its applications (Mesa,
  AZ, 1987), vol.~35, Birkh\"auser, 1988, pp.~249--263.

\end{thebibliography}
\end{document}